\documentclass[a4paper]{amsart}
\usepackage[margin=3.5cm]{geometry}
\usepackage{amsmath,amssymb,amsthm,amscd,amsfonts,comment,mathtools}

\usepackage[toc,page]{appendix}
\usepackage{hyperref}
\usepackage{enumerate}
\usepackage{nicematrix}
\usepackage{enumitem}

\usepackage{tikz}
\usetikzlibrary{cd}
\usetikzlibrary{fit, patterns} 
\usetikzlibrary{arrows,decorations.markings,plotmarks}

\DeclareMathOperator{\rk}{rk} 
\DeclareMathOperator{\pt}{pt}
\DeclareMathOperator{\id}{id}
\DeclareMathOperator{\Hom}{Hom}
\DeclareMathOperator{\Ext}{Ext}
\DeclareMathOperator{\Rep}{\mathrm{Rep}}
\DeclareMathOperator{\Mat}{\mathrm{Mat}}

\DeclareMathOperator{\codim}{\mathrm{codim}}
\DeclareMathOperator{\GL}{\mathrm{GL}}

\newcommand{\Hh}{H}
\newcommand{\Rr}{R}
\newcommand{\R}{\mathbb{R}}
\newcommand{\C}{\mathbb{C}}
\newcommand{\Q}{\mathbb{Q}}
\newcommand{\Z}{\mathbb{Z}}
\newcommand{\N}{\mathbb{N}}
\newcommand{\A}{\mathbb{A}}
\newcommand{\CO}{\mathcal{O}}
\newcommand{\CR}{\mathcal{R}}
\newcommand{\CM}{\mathcal{M}}
\newcommand{\CP}{\mathcal{P}}
\newcommand{\CA}{\mathcal{A}}
\newcommand{\ur}{\underline{r}}
\newcommand{\um}{\underline{m}}
\newcommand{\us}{\underline{s}}

\newcommand{\ud}{\underline{d}}
\newcommand{\uv}{\underline{v}}
\newcommand{\up}{\underline{p}}
\newcommand{\ue}{\underline{e}}

\newcommand{\Tr}{\mathrm{Tr}}

\newcommand{\rlct}{\mathrm{rlct}}
\newcommand{\rlcm}{\mathrm{rlcm}}
\newcommand{\ins}{\mathrm{in}}
\newcommand{\out}{\mathrm{out}} 
\newcommand{\DLN}{\mathrm{DLN}}
\newcommand{\orb}{\mathrm{orb}}
\newcommand{\Sym}{S}  
\newcommand{\mult}{\mathrm{mult}}

\newcommand{\Lace}{\mathrm{Lace}}

\newcommand{\cSigma}{\overline{\Sigma}}

\newcommand{\codimnot}{C}

\newtheorem{thm}{Theorem}[section]
\newtheorem{prop}[thm]{Proposition}
\newtheorem{lemma}[thm]{Lemma}
\newtheorem{defn}[thm]{Definition}
\newtheorem{cor}[thm]{Corollary}

\newtheorem{propdefn}[thm]{Proposition-Definition}
\theoremstyle{plain}
\newtheorem{example}[thm]{Example}
\newtheorem{remark}[thm]{Remark}

\title[Geometry of linear neural networks]{Geometry of the fibers of the multiplication map \\ of deep linear neural networks}
\author{Simon Pepin Lehalleur}
\address{Universiteit van Amsterdam, Netherlands}

\author{Rich\'{a}rd Rim\'{a}nyi}
\address{Department of Mathematics, UNC Chapel Hill, Chapel Hill, NC, USA}

\begin{document}

\begin{abstract}
We study the geometry of the algebraic set of tuples of composable matrices which multiply to a fixed matrix, using 
tools from the theory of quiver representations. In particular, we determine its codimension $\codimnot$ and the number $\theta$ of its top-dimensional irreducible components. Our solution is presented in three forms: a Poincar\'e series in equivariant cohomology, a quadratic integer program, and an explicit formula. In the course of the proof, we establish a surprising property: $\codimnot$ and $\theta$ are invariant under arbitrary permutations of the dimension vector. We also show that the real log-canonical threshold of the function taking a tuple to the square Frobenius norm of its product is $\codimnot/2$. These results are motivated by the study of deep linear neural networks in machine learning and Bayesian statistics (singular learning theory) and show that deep linear networks are in a certain sense ``mildly singular".
\end{abstract}

\maketitle

\section{Introduction}

In this paper we study the possible ways for a tuple of matrices to multiply to zero (or more generally a fixed matrix). This elementary linear algebra problem leads to some rich geometry and combinatorics, and has applications to deep linear neural networks in statistics and machine learning.

\smallskip

\subsection{Mathematical setting}

Fix nonnegative integers $d_0,d_1,\ldots, d_N$ and consider the vector space $\Rep=\Rep_{\ud}$, consisting of tuples of matrices $A_1\in k^{d_1\times d_0}$, $A_2\in k^{d_2\times d_1},$ $\ldots,$ $A_N\in k^{d_N\times d_{N-1}}$, where $k$ is an arbitrary field (say, $\R$ or $\C$ for the purpose of this introduction). Define $\Sigma=\Sigma^0_{\ud}$ as the subvariety of $\Rep$ consisting of tuples whose product is zero:
\[
\Sigma=
\{ (A_1,A_2,\ldots,A_N)\in \Rep \ :\ A_N A_{N-1} \cdots A_2 A_1=0\}.
\]
The variety $\Sigma$ generally has many irreducible components. For example, when $d_0=d_1=d_2=1$, we have two components: $\{A_1=0\}$ and $\{A_2=0\}$. For more complex sequences of $d_i$, the variety $\Sigma$ typically has numerous components of varying dimensions (see Example~\ref{ex:222324}).

Describing $\Sigma$'s components or even the multiset of their dimensions seems to be a challenging combinatorial problem. In this paper, we determine the codimension $\codimnot=\codim_{\Rep}\Sigma$ of $\Sigma$ in $\Rep$---or equivalently the codimension of its largest dimensional irreducible components---and the number $\theta$ of such top-dimensional components. 

To study $\Sigma$, we borrow ideas and results from the theory of quiver representations. 
The vector space $\Rep$ has natural symmetries coming from change of bases, which yields an action of the group $G=\prod_{i=0}^{N} \GL_{d_i}$. The starting point of our proofs is that this action of $G$ on $\Rep$ has finitely many orbits and that the orbits are classified by simple combinatorial data called \emph{Kostant partitions} (Corollary \ref{cor:gabriel}). This is a special case of Gabriel's theorem on representations of Dynkin quivers. Moreover, $\Sigma$ is $G$-invariant and so it is the union of some of those orbits. Our main task is thus to understand the geometry and combinatorics of which $G$-orbits occur in $\Sigma$.

In fact, throughout the paper we study and solve the same questions for the variety $\Sigma^{r}_{\ud}$ of tuples whose product has rank equal to a given number~$r$, as well as for the variety $\mult^{-1}(B)$ of tuples whose product is a fixed matrix $B$. Here, $\mult$ is the multiplication map.
\[
\mult:\Rep_{\ud}\to \Mat_{d_N,d_0}, (A_1,\ldots,A_N)\mapsto A_N\cdot A_{N-1}\ldots A_1
\]
Those problems easily reduce to the case $r=0$ and the variety $\Sigma$ above (Lemmas \ref{lem:rank_0} and \ref{lem:rank_vs_fibers}).

\subsection{The results}
The main results of the paper are formulas for $\codimnot$ and $\theta$. We have three in the main body of the paper, (Theorems~\ref{thm:Pseries}, \ref{thm:QIP}, and \ref{thm:main-codim}) and one in the appendix (Theorem~\ref{thm:UsingAvoidingIdeal}).

Theorem~\ref{thm:Pseries} uses Poincar\'e series calculations in equivariant cohomology, and produces an explicit formal power series in $q$ whose lowest-degree term is $\theta q^\codimnot$. The resulting formula implies, in particular, that neither $\codimnot$ nor $\theta$ depend on the ordering of the integers $d_0,d_1,\ldots,d_N$. This is a priori surprising, since reordering the $d_i$'s can change the dimension of $\Rep$ and drastically alter both the number of components of $\Sigma$ and the multiset of their codimensions. We expect there should be an elementary algebraic proof of this invariance, without topology. However, we do not provide such a proof here, relying instead on Theorem~\ref{thm:Pseries} (an alternative argument, still relying on equivariant cohomology, is given in Appendix~\ref{sec:app_theta}). A key component of our Poincar\'e series calculation is a remarkable result from Gabriel's theory of quivers: the indecomposable module corresponding to the longest root in equioriented type A quivers is both projective and injective.

In Theorem~\ref{thm:QIP}, we present $\codimnot$ as the optimal solution of a quadratic integer program and $\theta$ as the number of such optimal solutions. The proof relies on the combinatorics and algebra of quivers, as well as the crucial invariance of $\codimnot$ and $\theta$ with respect to the ordering of the integers $d_0,d_1,\ldots,d_N$, as discussed earlier.

In Theorem~\ref{thm:main-codim}, we provide our most explicit formulas for $\codimnot$ and $\theta$, which are well-suited for asymptotic analysis. The proof hinges on reinterpreting the quadratic integer program as a minimal distance problem: finding the closest lattice point in a simplex to a given point in Euclidean space. Central to this problem is the solution to the closest vector problem for type A root lattices, where the Voronoi cells of these lattices have the expected structure.

\subsection{Deep linear networks in Bayesian statistics and Machine Learning}\label{sec:DLN_intro}

One motivation for studying the geometry of composable tuples of matrices is \emph{deep linear (neural) networks} and specifically their statistical properties in Bayesian statistics. We give a quick overview here and refer to the introduction of Section \ref{sec:RLCT} for more context.

Deep linear networks (DLNs) are obtained from standard (feedforward, fully-connected) deep neural networks by replacing their non-linear activation functions by identity maps. Despite their simplicity, they are a useful ``toy model" in modern deep learning theory (see Section \ref{sec:related}). By definition, the weights of a DLN form a tuple of composable matrices, and the function computed by the network is their product, so the parameter space of the model is our friend $\Rep_{\ud}$ where $\ud\in \N^{N+1}$ records the widths of the layers and $N$ is the depth of the network. 

A fundamental problem in statistics and machine learning is \emph{density estimation}: given a fixed parametric statistical model and data generated from an unknown ``true" probability distribution, infer the ``optimal" parameter(s), for which the model is as close as possible to the true distribution. A natural way to measure the quality of this approximation is the relative entropy (or Kullback-Leibler divergence) $K(w)$ between the true distribution and the model, considered as a function of the parameters $w$. For regression models such as neural networks, $K(w)$ is the \emph{mean square error population loss function}. Classical statistical learning theory often focuses on \emph{regular} statistical models, in which there is a unique optimal parameter $w^*$ where $K(w)$ has a unique non-degenerate minimum. However, deep neural networks (linear or not) and other parametric models in modern machine learning are almost always \emph{singular}. For a DLN where the true distribution is given by a matrix $B$, we write $K^{\DLN}_B:\Rep\to \R$ for the corresponding relative entropy function:
\[
K^{\DLN}_B(A_*)=\| \mult(A)-B \|^2_2.
\]
The set of optimal parameters is $(K^{\DLN}_B)^{-1}(0)=\mult^{-1}(B)$, so the results of this paper precisely describe the geometry of the bottom of the loss landscape of DLNs.

Singular Learning Theory, a theory established by S. Watanabe \cite{watanabe2009,watanabe2018,watanabe2024recent}, studies the asymptotic performance of (real analytic) singular models for Bayesian density estimation. A key role is played by the \emph{real log-canonical threshold} $\rlct(K)$ of the real analytic function $K$ at its optimal parameters. Singular learning theory also provides a tool to \emph{estimate} $\rlct(K)$ from data, the \emph{local learning coefficient} \cite{lau2024locallearningcoefficientsingularityaware}. The local learning coefficient has recently lead to very interesting applications to developmental interpretability of deep learning models throughout training \cite{chen2023dynamicalversusbayesianphase,hoogland2024developmentallandscapeincontextlearning,wang2024differentiationspecializationattentionheads}.

The \emph{exact, theoretical} computation of $\rlct(K)$ in realistic machine learning scenarios is generally impossible. However, for deep linear networks, $\rlct(K)$ has been calculated by M.~Aoyagi \cite{aoyagi} and this was used in \cite{lau2024locallearningcoefficientsingularityaware} to calibrate the accuracy of the local learning coefficient. The original formula of Aoyagi is complicated, and we started this project in order to understand it better. Combined with our results, it takes a rather simple form (Theorem \ref{thm:aoyagi-rlct}):
\[\rlct(K^{\DLN}_B)=\frac{\codim (K^{\DLN}_B)^{-1}(0)}{2}=\frac{\codim\mult^{-1}(B)}{2}.\]
For a general real analytic function $F$, we only have the inequality 
\[
\rlct(F)\leq\frac{\codim F^{-1}(0)}{2}.
\]
This shows that the singularities of $K$ are in some sense quite mild. Statistically speaking, this means that deep linear networks are  ``mildly singular".

\subsection{Related works}\label{sec:related}

The geometry and combinatorics of representations of Dynkin quivers, and especially equioriented type A quivers, is a rich topic in combinatorial algebraic geometry and representation theory \cite{abeasis-del_fra,abeasis-del_fra-kraft,abeasis-del_fra:equioriented,FRduke,BFRpositivity,RWY}, including connections with Schubert varieties and {\em standard monomial theory} \cite{zelevinskii:schubert,lakshmibai-magyar:quiver-schubert,kinser:quiver-schubert,knutson-miller-shimozono}. These works mainly focus on the geometry of individual orbit closures and not on possibly reducible unions of orbits like $\Sigma$. Interesting exception are the papers \cite{concini-strickland:complexes,musili-seshadri:schubert-complexes} on the \emph{variety of complexes}, consisting of tuples of matrices such that the product of any two consecutive matrices is $0$.

Deep linear networks have been studied extensively in machine learning. The papers \cite{trager-kohn-bruna:DLN-spurious-critical} and \cite{bhattacharya-shewchuk} are closely related to ours. In \cite{trager-kohn-bruna:DLN-spurious-critical}, the authors start the study of the geometry of the fibers $\mult^{-1}(B)$. The paper \cite{bhattacharya-shewchuk} goes much further, including proofs using linear algebra of a number of results on the geometry of $\mult^{-1}(B)$ which in the perspective of this paper are part of the basic representation theory of the quiver $Q$ (which we collected in Sections \ref{secQG}, \ref{sec:rank_patterns}, \ref{sec:mult}).

A large part of the literature on DLNs is concerned with other critical points of the loss functions \cite{kawaguchi:DLN_no_local_min,kawaguchi-lu:DLN_no_local_min_II,achour-al:DLN_landscape_Hessian} and with the training dynamics by gradient descent \cite{Liu-Li-Meng:exact_DLN,ji-telgarsky:GD-DLN,saxe-al:exact-dynamics-DLN,arora-al:convergence_GD_DLN,saxe-mcclellang-ganguli:DLN,jacot-al:saddle-to-saddle-DLN,marion-chizat:DLN-flat-minima}. Our results do not directly bear on these problems, but we hope that the quiver representation approach will prove helpful to study them.

\bigskip

\noindent{\bf Acknowledgments.}
This paper grew out of a collaboration that started when the authors visited the Isaac Newton Institute for Mathematical Sciences in the Spring of 2024. The authors are grateful to the organizers of the special semester on {\em New equivariant methods in algebraic and differential geometry}, which was supported by EPSRC grant no EP/R014604/1, as well as the Newton Institute for their hospitality. This first named author was supported by the European Research Council (ERC), grant 864145. The second named author was partially supported by NSF grant 2152309, and was partially supported by a grant from the Simons Foundation.

\section{Tuples of matrices as quiver representations}
\label{secQG}
\subsection{The equioriented quiver of type A}

Consider the oriented graph $Q$ 
\[
\begin{tikzpicture}
 \node (A0) at (0,0)   {$\bullet$};
 \node (A1) at (1,0)   {$\bullet$};
 \node (A2) at (2,0)   {$\bullet$};
 \node (A3) at (3,0)   {$\bullet$};
 \node at (4.5,0)   {$\ldots$};
 \node (A6) at (6,0)   {$\bullet$};
 \node (A7) at (7,0)   {$\bullet$};
 \draw[->] (A0) to (A1);
 \draw[->] (A1) to (A2);
 \draw[->] (A2) to (A3);
 \draw[->] (A3) to (4,0);
 \draw[->] (5,0) to (A6);
 \draw[->] (A6) to (A7);
 \node at (0,.35) {$0$};
 \node at (1,.35) {$1$};
 \node at (2,.35) {$2$};
 \node at (3,.35) {$3$};
 \node at (6,.35) {$N\!-\!1$};
 \node at (7,.35) {$N$};
\end{tikzpicture},
\]
 with vertices labeled by $0,1,\ldots,N$. The length $N$ will be fixed for the whole paper. In representation theory, such an oriented graph is called a {\em quiver}. The quiver $Q$ (the ``type~A equioriented quiver") is the only one we will study in this paper.

In the next few sections~\ref{secQG},~\ref{sec:rank_patterns} and~\ref{sec:mult} we will recall the relevant geometry, algebra, and combinatorics for $Q$, with general reference \cite[Ch.~1-3]{kirillov}. We do not claim to any originality until Section \ref{sec:PoincareSection}. While many results in those sections extend in a straightforward manner to arbitrary Dynkin graphs with arbitrary orientations. We do not seek that generality, since our main results in later sections only apply to $Q$. 

Equioriented type A quivers are among the simplest of all quivers (Dynkin or otherwise). Consequently, the results recalled in Sections~\ref{secQG},~\ref{sec:rank_patterns} and~\ref{sec:mult} can be proven without reference to the theory of quivers---just using elementary linear algebra, an approach adopted in \cite{bhattacharya-shewchuk}. However, the quiver perspective and appealing to Gabriel's theorem \ref{thm:type_A_indecomp} makes the structure of the arguments and the resulting combinatorics rather transparent, and connects our results to the rich literature on Dynkin quivers.

\begin{remark} \rm
    It is also worthwhile to point out that the underlying reason behind Theorem~\ref{thm:addlongest} (a key point of what follows) is that there exists an indecomposable $Q$-module that is {\em both injective and projective}, namely the module $M_{0N}$ where every vector space is $k$ and every map is $\id_k$. This does not occur for any other Dynkin quiver. Curiously, this fact is behind other geometric theorems in seemingly unrelated areas, see e.g. \cite[Rem. 4.6]{FRduke}. 
\end{remark}

\subsection{Quiver geometry}

Let $k$ be a field---we are mostly interested in $k=\R$ and $k=\C$. For our quiver $Q$ and a fixed \emph{dimension vector} $\ud=(d_0,d_1,\ldots,d_N)\in \N^{N+1}$, define 
\[
\Rep_{\ud}(k)=\prod_{i=1}^{N}\Mat_{d_{i},d_{i-1}}(k)
=
\prod_{a} \Hom(k^{t(a)},k^{h(a)})
\]
where $\Mat_{d,d'}(k)$ is the vector space of  $d\times d'$ matrices with entries in $k$ and $a=(t(a)\to h(a))$ runs through the arrows of $Q$. 

We consider $\Rep_{\ud}(k)$ as an algebraic variety over $k$, or, more precisely, as the set of $k$-points of the $k$-algebraic variety $\Rep_{\ud}=\prod_{i=1}^{N}\Mat_{d_{i-1},d_i}$, which is isomorphic to the affine space $\A^D_k$ with 
\begin{equation}\label{eq:dim_Rep}
   \dim_k \Rep_{\ud} =\sum_{i=1}^N d_{i-1}d_i. 
\end{equation}

When the role of the field $k$ is not important we will simply write $\Rep_{\ud}=\Rep_{\ud}(k)$.

\begin{remark} \label{rmk:choice_k}\rm
    In general, when $X$ is an algebraic variety defined over $\R$ (or over any non algebraically closed field), it is important to distinguish between $X$ and the \emph{real algebraic set} $X(\R)$ of real points of $X$. For instance, $X(\R)$ may well be empty. Think of $X=\{x^2=-1\}\subset \A^1_\R$.
    However, in the situation of interest in this paper, this distinction will be mostly immaterial, in a sense that is made precise in Section \ref{sec:choice_k} (see however Remark~\ref{rmk:real_points}). For this reason, we encourage readers less familiar with algebraic geometry over non-closed fields to ignore this distinction at first pass.
\end{remark}

The algebraic matrix group $G=G_{\ud}:=\prod_{i=0}^{N} \GL_{d_i}$ acts on the representation space $\Rep_{\ud}$ as follows: for $g=(P_0,\ldots,P_N)\in G$ and $A_*=(A_1,\ldots,A_n)\in \Rep_{\ud}$, we have
\[
g\cdot A_*:=(P_1 A_1 P_0^{-1}\ ,\ P_2 A_2 P_1^{-1}\ ,\ldots,\ P_N A_N P_{N-1}^{-1}).
\]

\begin{remark} \label{rem:GinGout} \rm
    The group $G$ has two natural subgroups 
\[
G_\ins:= \prod_{i=1}^{N-1} \GL_{d_i}\subset G
\qquad\qquad
G_\out:=\GL_{d_0}\times \GL_{d_N}.\]
We have $G=G_\ins\times G_\out$ and we write $\pi_\out:G\to G_\out, (P_0,\ldots,P_N)\mapsto (P_0,P_N)$ for the projection onto the outer factors. The action of $G_{\ins}$ corresponds to change of bases of the vector spaces $k^{d_i}$. Note that $G_\out$ also naturally acts on $\Mat_{d_N,d_0}$ by
\[
(P_0,P_N)\cdot B:= P_N B P_0^{-1}.
\]
\end{remark}

\subsection{Quiver algebra}\label{sec:quiver_alg}

A $Q$-module, or $Q$-representation, is a chain of finite dimensional vector spaces and linear maps 
\[
V_0 \xrightarrow{f_1} V_1 \xrightarrow{f_2} V_2 \xrightarrow{f_3} \ldots \xrightarrow{f_{N-1}} V_{N-1} \xrightarrow{f_N} V_N. 
\]
Its dimension vector is defined to be $\ud=(\dim(V_i))_{i=0,\ldots,N}$. In other words, $Q$-modules are to elements of $\Rep$ what linear maps between finite dimensional vector spaces are to matrices.

A morphism between $Q$-modules $(V_*,f_*)$ and $(V_*',f_*')$ is a collection of linear maps $\phi_i:V_i\to V'_i$ that commute with $f_*$ and $f'_*$:
\[
f'_i \circ \phi_{i-1}= \phi_i \circ f_i.
\]
Many standard notions of linear algebra (of vector spaces and modules over rings) extend to $Q$-modules and their morphisms, such as: exact sequence, isomorphism, direct sum, subspace, irreducibility (no non-trivial subspace), indecomposability (not a direct sum in a non-trivial way), $\Ext(-,-):=\Ext^1(-,-)$, injective module, projective module. Moreover, $\Ext$, injective and projective behave quite simply (compared e.g. to modules over general rings). 

In more abstract terminology, the category of $Q$-modules is a $k$-linear abelian category which is Hom-finite, hereditary (no higher Ext). Injective and projective objects can also be completely classified (we will not need this).

The reader may develop familiarity  with these concepts by verifying
\begin{itemize}
    \item that $k \xrightarrow{\id} k$ is indecomposable but not irreducible; 
    \item $\Ext((k\to 0), (0\to k))=k$; and
    \item $V_i=k$, $A_i=\id$ ($\forall i$) is both an injective $Q$-module and a projective $Q$-module. 
\end{itemize}

The Krull–Schmidt theorem holds for $Q$-modules: every $Q$-module is the direct sum of indecomposable ones, in a unique (up to permutation) way. 

\subsection{Quiver geometry vs algebra. Gabriel's theorem}

Notice that elements of $\Rep_{\ud}$ can be made into $Q$-modules with $V_i=k^{d_i}$. We call the following well known and easy statement a theorem because of its importance.

\begin{thm} \label{thm:easy}
\ 
\begin{itemize}
    \item Every $Q$-module is isomorphic with an element of $\Rep_{\ud}$ for some $\ud$.
    \item Two elements of $\Rep_{\ud}$ are isomorphic as $Q$-modules if and only if they are in the same $G$-orbit.
\end{itemize}
\end{thm}

Hence, listing the $G$-orbits of $\Rep_{\ud}$ is the same as listing the isomorphism types of $Q$-modules with dimension vector $\ud$. 

For a general quiver, this set of isomorphism types is infinite; for example, for the Jordan quiver with one vertex and one loop, this is precisely the set of normal forms for matrices under conjugation. However, for Dynkin quivers, this set is finite and can be described explicitly. This is the content of Gabriel's theorem which we state in our special case:

\begin{thm} \label{thm:type_A_indecomp}
    For $0\leq i \leq j \leq N$ let $M_{ij}$ be the $Q$-module
    \[
    0 \to 0 \to \ldots 0 \to 
    k \xrightarrow{\id} k \xrightarrow{\id} \ldots \xrightarrow{\id} 
    k
    \xrightarrow{\id} k \to 0 \to 0 \ldots 
    \to 0,
    \]
    where the first and last $k$'s are $V_i$ and $V_j$. The $M_{ij}$'s form a complete set of indecomposable $Q$-modules (up to isomorphism). 
\end{thm}


\begin{defn}
        Define 
\[
\CM_{\ud}:=\{\um=(m_{ij})_{0\leq i\leq j\leq N}\ |\ \forall i\leq j,\ m_{ij}\in \Z\text{ and }\forall k,\ d_k=\sum_{i\leq k\leq j} m_{ij}\},
\]
\[
\CM^+_{\ud}:=\{\um=(m_{ij})_{0\leq i\leq j\leq N}\ |\ \forall i\leq j,\ m_{ij}\in \N\text{ and }\forall k,\ d_k=\sum_{i\leq k\leq j} m_{ij}\}.
\]
Elements of the $\CM^+_{\ud}$ are called {\em Kostant partitions} for $\ud$. (We think of an $\um$ as an upper triangular matrix whose rows and columns are indexed by $0,\ldots,N$.). Note that the set $\CM^+_{\ud}$ of Kostant partitions is finite (unlike $\CM_{\ud}$).

We also use the notation $\um \vdash \ud$ for ``$\um$ is a Kostant partition of $\ud$''.
\end{defn}

\begin{example}
  The set $\CM^+_{(2,2,2)}$ has 6 elements:
  \[
\begin{pNiceArray}{ccc}
    2 & 0 & 0  \\
    0 & 2 & 0  \\
    0 & 0 & 2 \\
\end{pNiceArray}, 
\begin{pNiceArray}{ccc}
    1 & 1 & 0  \\
    0 & 1 & 0  \\
    0 & 0 & 2 \\
\end{pNiceArray},
\begin{pNiceArray}{ccc}
    2 & 0 & 0  \\
    0 & 1 & 1  \\
    0 & 0 & 1 \\
\end{pNiceArray},
\begin{pNiceArray}{ccc}
    0 & 2 & 0  \\
    0 & 0 & 0  \\
    0 & 0 & 2 \\
\end{pNiceArray},
\begin{pNiceArray}{ccc}
    2 & 0 & 0  \\
    0 & 0 & 2  \\
    0 & 0 & 0 \\
\end{pNiceArray},
\begin{pNiceArray}{ccc}
    1 & 1 & 0  \\
    0 & 0 & 1  \\
    0 & 0 & 1 \\
\end{pNiceArray}.
  \]

\end{example}

\begin{remark} \rm
    There are alternative terminologies for Kostant partitions in the literature on quiver representations: \emph{lace arrays}, \emph{multiplicity patterns}.
\end{remark}

\begin{cor} \label{cor:gabriel}
   The sets of the following objects are in bijection:
   \begin{itemize}
       \item isomorphism classes of $Q$-modules with dimension vector $\ud$;
       \item $G$-orbits of $\Rep_{\ud}$;
       \item Kostant partitions for $\ud$.
   \end{itemize}
   Explicitly, the Kostant partition $\um$ corresponds to the isomorphism class of the $Q$-module $\bigoplus_{0\leq i\leq j\leq N} m_{ij} M_{ij}$, and to the $G$-orbit $\CO_{\um}\subset \Rep_{\ud}$ consisting of all representations isomorphic to this $Q$-module.
\end{cor}

\begin{proof} Follows from the Krull-Schmidt theorem for $Q$-modules together with Theorems~\ref{thm:easy},~\ref{thm:type_A_indecomp}.
\end{proof}

To make this corollary concretely useful, we need to solve two basic problems:
\begin{itemize}
    \item Given a Kostant partition $\um\in \CM^+_{\ud}$, construct explicitly tuples in the $G$-orbit $\CO_{\um}\subseteq \Rep_{\ud}$.
    \item Given $A_*\in \Rep_{\ud}$, determine the Kostant partition corresponding to the orbit $G\cdot A_*$.
\end{itemize}
We solve the first one in the next section and the second in Section~\ref{sec:rank_patterns}.

\subsection{Lace diagrams and orbit representatives}
\label{sec:laces}


 Let $\um\in \CM^+_{\ud}$ be a Kostant partition. We want to produce an orbit representative in $\CO_{\um}$, i.e., construct an explicit tuple of matrices $A_*\in\Rep_{\ud}(k)$, with the right isomorphism class as a $Q$-module. It turns out that there are several natural choices, corresponding to combinatorial objects called {\em lace diagrams}.

A {\em lace diagram} is an arrangement of dots in columns $0,1,\ldots,N$, partitioned into a collection of $[i,j]$ intervals with $i\leq j$ (laces) - including intervals $[i,i]$, i.e. isolated points. For example
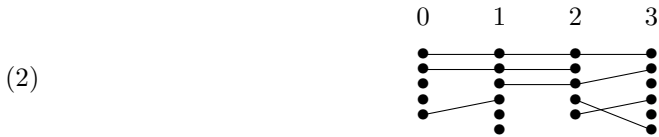
\begin{equation}
\label{eq:lace_diagram_bad}
\begin{tikzpicture}[baseline=-12]
\node at (0,.5) {$0$};
\node at (1,.5) {$1$};
\node at (2,.5) {$2$};
\node at (3,.5) {$3$};
 \node at (0,0)  {$\bullet$};
 \node at (0,-.2)  {$\bullet$};
 \node at (0,-.4)  {$\bullet$};
 \node at (0,-.6)  {$\bullet$};
 \node at (0,-.8)  {$\bullet$};

 \node at (1,0)  {$\bullet$};
 \node at (1,-.2)  {$\bullet$};
 \node at (1,-.4)  {$\bullet$};
 \node at (1,-.6)  {$\bullet$};
 \node at (1,-.8)  {$\bullet$};
\node at (1,-1)  {$\bullet$};

 \node at (2,0)  {$\bullet$};
 \node at (2,-.2)  {$\bullet$};
 \node at (2,-.4)  {$\bullet$};
 \node at (2,-.6)  {$\bullet$};
 \node at (2,-.8)  {$\bullet$};

 \node at (3,0)  {$\bullet$};
 \node at (3,-.2)  {$\bullet$};
 \node at (3,-.4)  {$\bullet$};
 \node at (3,-.6)  {$\bullet$};
 \node at (3,-.8)  {$\bullet$};
 \node at (3,-1)  {$\bullet$};

\draw (0,0) -- (3,0);
\draw (0,-.2) -- (2,-.2);
\draw (1,-.4) -- (2,-.4) -- (3,-.2);
\draw (0,-.8) -- (1,-.6);
\draw (2,-.6) -- (3,-1.0);
\draw (2,-.8) -- (3,-.6);
\end{tikzpicture}
\end{equation}
is a lace diagram with $N=3$. 

Write $\Lace_{\ud}$ for the set of all lace diagrams with $d_i$ dots in column $i$. A lace diagram $L\in\Lace_{\ud}$ encodes a Kostant partition $\um(L)$ of $\ud$ with $m(L)_{ij}$ the number of intervals of type $[i,j]$. For example, diagram \eqref{eq:lace_diagram_bad} encodes the Kostant partition
\[
\um=
\begin{pNiceArray}{cccc}
    2 & 1 & 1 & 1 \\
    0 & 2 & 0 & 1 \\
    0 & 0 & 0 & 2 \\
    0 & 0 & 0 & 2
\end{pNiceArray}_{0\ldots 3, 0\ldots 3}.
\]
As in Corollary \ref{cor:gabriel}, a lace diagram thus determines an isomorphism class $[\oplus_{ij}m_{ij} M_{ij}]$ of $Q$-modules. For example, the lace diagram \eqref{eq:lace_diagram_bad} above is associated to the class of the $Q$-module
    \[
    2M_{00} \oplus M_{01} \oplus M_{02} \oplus M_{03} 
    \oplus 
    2M_{11} \oplus 2 M_{13} 
    \oplus 
    2 M_{23} 
    \oplus
    2M_{33}.
    \]
    
Let $\Sym_k$ be the permutation group permuting $k$ letters. The group $\Sym_{\ud}:=\Sym_{d_0}\times \Sym_{d_2}\times \ldots \Sym_{d_N}$ acts on $\Lace_{\ud}$ by permuting the dots in each columns. The lace diagram above is in the orbit of 
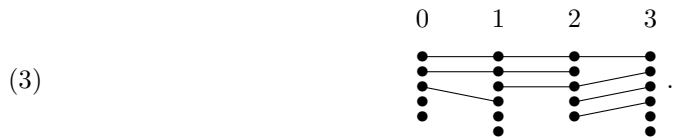
\begin{equation}
\label{eq:lace_diagram}
\begin{tikzpicture}[baseline=-12]
\node at (0,.5) {$0$};
\node at (1,.5) {$1$};
\node at (2,.5) {$2$};
\node at (3,.5) {$3$};
 \node at (0,0)  {$\bullet$};
 \node at (0,-.2)  {$\bullet$};
 \node at (0,-.4)  {$\bullet$};
 \node at (0,-.6)  {$\bullet$};
 \node at (0,-.8)  {$\bullet$};

 \node at (1,0)  {$\bullet$};
 \node at (1,-.2)  {$\bullet$};
 \node at (1,-.4)  {$\bullet$};
 \node at (1,-.6)  {$\bullet$};
 \node at (1,-.8)  {$\bullet$};
\node at (1,-1)  {$\bullet$};

 \node at (2,0)  {$\bullet$};
 \node at (2,-.2)  {$\bullet$};
 \node at (2,-.4)  {$\bullet$};
 \node at (2,-.6)  {$\bullet$};
 \node at (2,-.8)  {$\bullet$};

 \node at (3,0)  {$\bullet$};
 \node at (3,-.2)  {$\bullet$};
 \node at (3,-.4)  {$\bullet$};
 \node at (3,-.6)  {$\bullet$};
 \node at (3,-.8)  {$\bullet$};
 \node at (3,-1)  {$\bullet$};

\draw (0,0) -- (3,0);
\draw (0,-.2) -- (2,-.2);
\draw (1,-.4) -- (2,-.4) -- (3,-.2);
\draw (0,-.4) -- (1,-.6);
\draw (2,-.6) -- (3,-.4);
\draw (2,-.8) -- (3,-.6);
\end{tikzpicture}.
\end{equation}

Kostant partitions can be thought of as equivalence classes of lace diagrams:

\begin{lemma}\label{lem:Kostant_lace}
 The map $\um(-):\Lace_{\ud}\to \CM^+_{\ud}$ is invariant under the action of $\Sym_{\ud}$ and induces a bijection
    \[
    \Sym_{\ud}\backslash\Lace_{\ud}\simeq \CM^+_{\ud}.
    \]
\end{lemma}
\begin{proof}
    The map $\um(-)$ is clearly invariant under the action and thus induces a map 
    \[\Sym_{\ud}\backslash\Lace_{\ud}\to \CM^+_{\ud}.\] 
    
    Given a Kostant partition, one can construct a lace diagram by adding inductively intervals of length $m_{ij}$ to the diagram in any order, using the first dots still available in each column. This shows that the map is surjective. Given two lace diagrams with the same Kostant partition, one can first choose a length-preserving bijection between their intervals, and then permute the dots according to that bijection. This shows that the map is injective.
\end{proof}


We can interpret the choice of a lace diagram as the choice of a base point for the associated orbit.

\begin{defn}
    Let $L$ be a lace diagram. Identify the $N+1$ columns of dots of $L$ as the standard basis vectors in the $N+1$ vector spaces $k^{d_0},\ldots,k^{d_N}$. The line segments of $L$ then encode linear transformations between these spaces by specifying how a basis vector is mapped to another basis vector to the right. If a dot has no segment coming out of it to the right, then that basis vector is mapped to 0. The corresponding matrices yield an element $A_*(L)\in\CO_{\um(L)}\subset\Rep_{\ud}$. 
\end{defn}

Note that the resulting matrices are very special: they are {\em partial permutation matrices}:
\begin{itemize}
    \item all coefficients are $0$ or $1$, and
    \item every line and column has at most one $1$.
\end{itemize}

The combinatorics of lace diagrams are quite rich (see for instance \cite{BFRpositivity, knutson-miller-shimozono,RWY}). In this paper, we use them in an elementary way in the proof of Theorem \ref{thm:QIP}, see in particular Lemma \ref{lem:horiz_rep} and Figure \ref{fig:largeLD}.

\section{Rank patterns and combinatorics of orbits}
\label{sec:rank_patterns}

\subsection{Rank patterns}
Kostant partitions are the natural combinatorial codes for $G$-orbits of $\Rep_{\ud}$---in fact this holds for an arbitrary oriented Dynkin quiver if we replace ``intervals'' with ``positive roots of the same named root system'' in Theorem \ref{thm:type_A_indecomp}. However, for equioriented type A quivers there is another natural combinatorial code: rank patterns. Define 
\[\CR_{\ud}:=\{\ur=(r_{ij})_{0\leq i\leq j\leq N}\ |\ \forall i\leq j, r_{ij}\in\N\text{ and }r_{ii}=d_i\}
\]
to be the set of upper triangular arrays of size $N+1$ with non-negative integer entries, whose diagonal encodes our fixed dimension vector. We call elements $\ur\in \CR_{\ud}$ \emph{rank patterns}. 

\begin{prop} \ \label{prop:mr_comparison}
    \begin{itemize}
        \item The maps
        $\CR_{\ud} \to \CM_{\ud}$ and $\CM_{\ud} \to \CR_{\ud}$ defined by
        \[
        \begin{array}{ll}
        m_{ij}(\ur)= & r_{ij}-r_{i,j+1}-r_{i-1,j}+r_{i-1,j+1} \\ & (\text{with the convention } r_{ij}=0 \text{ if  } i < 0 \text{ or } j > N),\\
        r_{ij}(\um)= & \sum_{k\leq i\leq j\leq l}m_{kl} 
        \end{array}
        \]
        are well defined, and are inverses of each other.
        \item Let $\um$ and $\ur$ correspond to each other at the bijection above, and assume that $\um\in\CM_{\ud}^+$. The $G$-orbit in $\Rep_{\ud}(k)$ corresponding to $\um\in \CM^+_{\ud}$ is 
        \[
\CO_{\rk=\ur}:=\{(A_1,\ldots,A_N)\in \Rep_{\ud}\ |\ \forall\ 0\leq i<j\leq N,\ \mathrm{rank}(A_{j}\cdots A_{i+1})=r_{ij}\}.
\]  
    \end{itemize}
\end{prop}

That is, the $G$-orbits of $\Rep_{\ud}$ are determined by rank conditions of the possible compositions in $Q$, and the rank conditions which occur are exactly the ones that correspond to $\CM_{\ud}^+$ under the bijection. We leave the easy combinatorial proof to the reader. Write
\[
\CR^{\orb}_{\ud}=\{\ur\in \CR_{\ud}\ |\ \forall i,j,\ r_{ij}-r_{i,j+1}-r_{i-1,j}+r_{i-1,j+1}\geq 0\}
\]
for the set of rank patterns corresponding to Kostant partitions and hence giving an alternative parametrization of orbits.

The characterization of the rank in terms of vanishing of minors, which are polynomials in the coefficients, implies that the rank is lower semi-continuous in the Zariski topology and so that $\CR_{\ur}$ is a Zariski locally closed subvariety of $\Rep_{\ud}$. Note that the condition defining $\CO_{\rk=\ur}$ does not involve the diagonal $r_{ii}=d_i$, which is only used to encode the dimension vector $\ud$ into $\ur$ and to make some formulas cleaner.

\begin{example} \rm \label{Ex:first_example} 
The rank pattern corresponding to the lace diagram \eqref{eq:lace_diagram_bad} and the Kostant partition in Section \ref{sec:laces} is 
\[
\ur=
\begin{pNiceArray}{cccc}
    5 & 3 & 2 & 1 \\
    0 & 6 & 3 & 2 \\
    0 & 0 & 5 & 4 \\
    0 & 0 & 0 & 6
\end{pNiceArray}_{0\ldots 3, 0\ldots 3}.
\]
The orbit $\CO_{\rk=\ur}$ is not empty, of dimension $72$ (as follows from Equation \ref{eqn:dim_formula} below). If we changed $r_{13}=2$ to $r'_{13}=1$ in this rank pattern (recall that the row and column indices run from $0$ to $3$), then it would not correspond to a Kostant partition since we would then have
\[
r_{22}-r_{23}-r_{12}+r'_{13}=5-4-3+1<0
\]
and hence we would have $\CO_{\rk=\ur}=\emptyset$.
\end{example}

\medskip

\noindent{\em  Notation.} Since Kostant partitions and corresponding rank patterns are equivalent combinatorial codes; both parametrize orbits. When we name an orbit with its rank pattern, we use the notation $\CO_{\rk=\ur}$, when we name an orbit with the Kostant partition $\um$ we will simply call it $\CO_{\um}$.

\subsection{The choice of base field}\label{sec:choice_k}

We clarify Remark \ref{rmk:choice_k} about the choice of base field $k$, and show in particular that the geometry we are interested is essentially the same over $k=\R$ and $k=\C$ (see however Remark~\ref{rmk:real_points} for an interesting difference). In this section only, we let $\Rep_{\ud}$ (resp. $\CO_{\rk=\ur}$) stand for the corresponding $k$-variety and distinguish it from its set of rational points $\Rep_{\ud}(k)$ (resp. $\CO_{\rk=\ur}(k)$).

\begin{thm}\ \label{thm:base_field}
\begin{enumerate}
        \item\label{equiv_empty} For any field $k$ and $\ur\in \CR_{\ud}$, we have
        \[
        \ur\in \CR^\orb_{\ud}
        \quad \Leftrightarrow\quad
        \CO_{\rk=\ur}\neq\emptyset
        \quad \Leftrightarrow\quad
        \CO_{\rk=\ur}(k)\neq \emptyset.
        \]
        \item For every $\ur\in \CR^{\orb}_{\ud}$, the $k$-rational points $\CO_{\rk=\ur}(k)$ are Zariski dense in $\CO_{\rk=\ur}$.
        
        \item When $k=\R$ (resp. $k=\C$), the real algebraic set $\CO_{\rk=\ur}(\R)$ is a real analytic (resp. complex analytic) manifold of dimension equal to the dimension of $\CO_{\rk=\ur}$ as an algebraic variety.
     \end{enumerate}
\end{thm}
\begin{proof}
We prove Claim \eqref{equiv_empty}. The implication $ \CO_{\rk=\ur}\neq \emptyset \Leftarrow \CO_{\rk=\ur}(k)\neq \emptyset$ is obvious. If $\CO_{\rk=\ur}\neq \emptyset$, then $\CO_{\rk=\ur}(\bar{k})\neq \emptyset $ so by Proposition \ref{prop:mr_comparison} applied over $\bar{k}$ we have $ \ur\in \CR^\orb_{\ud}$. Finally, if $ \ur\in \CR^\orb_{\ud}$, then $\um(\ur)\in \CM^{+}_{\ud}$ by Proposition \ref{prop:mr_comparison}. Then by choosing a lace diagram $L$ encoding $\um(\ur)$ and following Section \ref{sec:laces}, we obtain a collection of partial permutation matrices $A_*(L)$ in the $G$-orbit corresponding to $\um(\ur)$. Since $0,1\in k$, this shows $\CO_{\rk=\ur}(k)=\CO_{\um(\ur)}(k)\neq \emptyset$.

Claim \eqref{equiv_empty} then implies that $\CO_{\rk=\ur}\simeq G/H$ as $k$-varieties, where $H$ is the stabilizer group of the orbit. This implies the other statements and finishes the proof.
\end{proof}

\subsection{Voight lemma and the codimension of orbits}\cite{voight}.  \label{sec:voight} 
Another crucial relation between the geometry and algebra of quivers is

\begin{lemma}
    \label{lem:voight}
    Let $M$, considered as a $Q$-module, be an element of a $G$-orbit $\CO \subset \Rep_{\ud}(k)$. Then a normal slice to $\CO$ at $M$ in $\Rep_{\ud}(k)$ is isomorphic to $\Ext(M,M)$.
\end{lemma} 

While the isomorphism holds as an isomorphism of representations of the $G$-stabilizer group of $M$, we will only use the statement as isomorphism of vector spaces. 

The first consequence of Lemma \ref{lem:voight} is that a formula for the codimension of an orbit is a (non-symmetric) bilinear form, in the components of the Kostant partition. 
\begin{cor}
    Let $\um\in\CM_{\ud}^+$ be a Kostant partition and $\ur=\ur(\um)\in \CR^\orb_{\ud}$ the corresponding rank pattern. Then
    \begin{equation}\label{eqn:dim_formula}
    \codim_{\Rep}(\CO_{\um})=\sum_{1\leq i\leq u\leq j\leq v\leq N}m_{(i-1)(j-1)}m_{uv} 
    \end{equation}
\end{cor}
\begin{proof}
Let $M$ be a $Q$-module in $\CO_{\um}$ then by Lemma~\ref{lem:voight} we have
\begin{multline*}
\codim_{\Rep}(\CO_{\um})
=
\dim \Ext(M,M)
=
\dim \Ext( \oplus_{ij} m_{ij}M_{ij}, \oplus_{uv} m_{uv}M_{uv})
\\
=
\sum_{ij}\sum_{uv} m_{ij}m_{uv}\dim \Ext(M_{ij},M_{uv}),
\end{multline*}
where an easy calculation (cf. \cite[Lemma~4.4]{FRduke}) gives 
\begin{equation}\label{eqn:ExtIndecomposables}
\dim \Ext(M_{ij},M_{uv})=
\begin{cases}
    1 & \text{if } i+1\leq u \leq j+1 \leq v, \\
    0 & \text{otherwise.}
\end{cases}
\end{equation}
The formula in terms of $\ur=\ur(\um)$ follows from the definition of $\ur(-)$ (Proposition~\ref{prop:mr_comparison}).
\end{proof}

\begin{remark} \rm
Alternatively, one can explicitly compute the Lie algebra of the stabilizer group of a point in $\CO_{\rk=\ur}$ and deduce equation \eqref{eqn:dim_formula} \cite[Lemma 3.2]{abeasis-del_fra-kraft}.
\end{remark}



\subsection{The effect of additional longest roots} \label{sec:add_longest}

For a dimension vector $\ud$, or a rank pattern $\ur$ let $\ud+p$, $\ur+p$ be defined by adding $p$ to every component. Consider $\CO_{\rk=\ur} \subset \Rep_{\ud}$, $\CO_{\rk=\ur+p} \subset \Rep_{\ud+p}$. Note that the Kostant partition corresponding to $\ur+p$ is obtained from the Kostant partition $\um$ of $\ur$ by adding $p$ to $m_{0N}$ and not changing the other components. In terms of lace diagrams, this consists of adding $p$ longest intervals.

\begin{thm}
\label{thm:addlongest}
    The normal slices of $\CO_{\rk=\ur}$ and $\CO_{\rk=\ur+p}$ are isomorphic. In particular, the codimensions of these two orbits (in different representation spaces) are equal.
\end{thm}

\begin{proof}
Using Lemma~\ref{lem:voight}, we have that the normal slice of $\CO_{\rk=\ur+p}$ is 
    \begin{multline}
    \Ext(M\oplus pM_{0N},M\oplus pM_{0N})\\=
        \Ext(M,M) \oplus
        p\Ext(M,M_{0N}) \oplus
        p\Ext(M_{0N},M) \oplus
        p^2\Ext(M_{0N},M_{0N})
        \\=
       \Ext(M,M),
        \end{multline}
    which is the normal slice to $\CO_{\rk=\ur}$.
    The last equality holds because $M_{0N}$ is both an injective and a projective $Q$-module \cite{FRduke}.
\end{proof}

\subsection{Orbit hierarchy}\label{sec:hierarchy}
We equip $\CR_{\ud}$ with the partial order induced by entry-wise comparison: for $\ur,\us\in \CR_{\ud}$, we write
\[
\us\leq \ur
\qquad
\stackrel{\text{def}}{\Leftrightarrow}
\qquad
\forall\ i\leq j,\ s_{ij}\leq r_{ij}. 
\]
In particular, $\underline{0}$ is the minimal element of $\CR_{\ud}$ for this partial order.  The lower semi-continuity of the rank implies that $\CO_{\rk=\us}\subset \overline{\CO}_{\rk=\ur} \Rightarrow \us\leq \ur$. The converse also holds, and can be established using the combinatorics of lace diagrams:

\begin{thm}{\cite{abeasis-del_fra:equioriented}}\label{thm:closures}
    If $\ur,\us\in \CR^\orb_{\ud}$, then $\CO_{\rk=\us}\subset \overline{\CO}_{\rk=\ur}$ if and only if $\us\leq \ur$. 
\end{thm}

\begin{remark} \rm
Theorem \ref{thm:closures} is also proven in \cite{bhattacharya-shewchuk}.
It is generalized to the case of non equioriented type A quivers in \cite{abeasis-del_fra} and to equioriented type D quivers in \cite{ABEASIS198481}. The remaining Dynkin cases seem to be open.
\end{remark}

 The algebraic geometry and the singularities of the orbit closures $\overline{\CO}_{\rk=\ur}$ have been studied extensively, see Section \ref{sec:related} for references.
 


\section{Multiplication map and related loci}\label{sec:mult}

Consider the multiplication map
\[
\mult:\Rep_{\ud}\to \Mat_{d_{N},d_0}, (A_1,\ldots,A_N)\mapsto A_N\cdot A_{N-1}\ldots A_1
\]
We are interested in the geometry of the map $\mult$ and its fibers. Note that, for $g=(P_0,\ldots,P_N)\in G$ and $A_*\in \Rep_{\ud}$, we have
\begin{equation}\label{mu_equiv}
    \mult(g\cdot A_*)= (P_N A_N P^{-1}_{N-1})\cdot (P_{N-1} A_{N-1} P^{-1}_{N-2})\ldots (P_1 A_1 P_0^{-1})= \pi_\out(g)\cdot \mult(A_*)
\end{equation}
In other words, the map $\mult$ is \emph{equivariant} with respect to the actions of $G$ on the source (resp. $G_\out$ on the target) and the homomorphism $\pi_\out:G\to G_\out$. In particular, when $g\in G_\ins$, we have $\mult(g\cdot A_*)=\mult(A_*)$. The $G_\out$-orbits on $\Mat_{d_{N},d_0}$ are parametrized by {\em rank}, which suggests the following definition.

\begin{defn}
For $r\in \N$, define
\[
\Sigma_{\ud}^r:=\{A_*\in\Rep_{\ud} |\ \rk \mult(A_*)= r\}.
\]
and 
\[
\cSigma_{\ud}^r:=\{A_*\in\Rep_{\ud} |\ \rk \mult(A_*)\geq r\}
\]
\end{defn}

\begin{lemma}\label{lem:sigma_non_empty}
   $\Sigma_{\ud}^r\neq\emptyset$ if and only if $0\leq r\leq \min\ud$.
\end{lemma}
\begin{proof}
    The direct implication is clear. For the converse, consider a lace diagram $L$ with $r$ horizontal intervals at the top. Then $A_*(L)\in \Sigma_{\ud}^r$.
\end{proof}

We are interested in the algebraic geometry of $\Sigma_{\ud}^r$ and $\cSigma_{\ud}^r$.

\begin{example} \rm \label{ex:222324}
 The sets $\cSigma^r_{\ud}$ are a natural object in linear algebra.
 \begin{itemize}
     \item $\cSigma^r_{a,b}$ is a \emph{determinantal variety}.
     \item $\Sigma^0_{(2,2,2)}$ is the collection of pairs of $2\times 2$ matrices $(A,B)$ for which $BA=0$. This variety turns out to have three components: (i) $\{A=0\}$ (of codimension 4), (ii) $\{B=0\}$ (of codimension 4), (iii) $\{\det(A)=\det(B)=0,BA=0\}$ (of codimension~3).
     \item $\Sigma^0_{(2,3,2)}$ is the collection of pairs of matrices $(A\in k^{3\times 2},B\in k^{2\times 3})$ for which $BA=0$. It has two components: (i) $\{\rk(A)\leq 1, BA=0\}$, (ii) $\{\rk(B)\leq 1, BA=0\}$. Both have codimension 4.
     \item  $\Sigma^0_{(2,4,2)}$ is the collection of pairs of matrices $(A\in k^{4\times 2},B\in k^{2\times 4})$ for which $BA=0$. It is irreducible of codimension 4.
     \item Let $\ud=(5,5,6,6,6,6)$. The variety $\Sigma^0_{\ud}$ has five codimension 19 components and many smaller dimensional components. One of the top dimensional components is 
     \begin{multline*}
        \ \hskip 1 true cm \{(A_1,A_2,A_3,A_4,A_5): \rk(A_1)\leq 3, \rk(A_2)\leq 4,\rk(A_4)\leq 4,\\ 
          \rk(A_2A_1)\leq 2,\rk(A_4A_3A_2)\leq 2,A_4A_3A_2A_1=0\}.
      \end{multline*}
      This, and other similar examples, can be verified using Theorem~\ref{thm:QIP} below. 
     \end{itemize}
\end{example}

By equivariance of $\mult$, we have stratifications by $G$-orbits
\[
\Sigma^r_{\ud}=\coprod_{\ur\in \CR_{\ud}^r}\CO_{\rk=\ur}=\coprod_{\um\in \CM_{\ud}^r}\CO_{\um}
\]
and
\[
\cSigma^r_{\ud}=\coprod_{\ur\in \CR_{\ud}^{\leq r}}\CO_{\rk=\ur}=\coprod_{\um\in \CM_{\ud}^{\leq r}}\CO_{\um}
\]
where
\[
\CR_{\ud}^{r}:=\{\ur\in \CR^\orb_{\ud}\ |\ r_{0N}= r\},\quad \CR_{\ud}^{\leq r}:=\{\ur\in \CR^\orb_{\ud}\ |\ r_{0N}\leq r\}
\]
and 
\[
\CM_{\ud}^r:=\{\um\in \CM^+_{\ud}\ |\ m_{0N}=r\},\quad\CM_{\ud}^{\leq r}:=\{\um\in \CM^+_{\ud}\ |\ m_{0N}\leq r\}.
\]

Theorem \ref{thm:closures} then implies

\begin{cor}\label{cor:irred_comp}
Let $0\leq r\leq \min\ud$.
\begin{enumerate}[label=(\alph*)]
    \item $\cSigma^r_{\ud}$ is the Zariski closure of $\Sigma^r_{\ud}$.
    \item The irreducible components of $\cSigma^r_{\ud}$ are in bijection with the minimal elements of $\CR_{\ud}^{\leq r}$ (or $\CM_{\ud}^{\leq r}$) under the partial order introduced in Section \ref{sec:hierarchy}.
\end{enumerate}
\end{cor}

We can also reduce the study of the combinatorics of $\cSigma^r_{\ud}$ to the case $r=0$ for a different dimension vector.

\begin{lemma}\label{lem:rank_0}
   Let $\ud\in \N^{N+1}$ and $0\leq r\leq \min\ud$. Then the injective map 
   \[
   \Sigma^{0}_{\ud-r}\hookrightarrow \cSigma^r_{\ud}, C_*\mapsto \left(\begin{array}{cc}
       I_r & 0 \\
      0  & C_*
   \end{array} \right)
   \]
   induces a bijection of irreducible components and we have
   \[
   \codim_{\Rep_{\ud}} \cSigma^{r}_{\ud}=\codim_{\Rep_{\ud}} \Sigma^{r}_{\ud}= \codim_{\Rep_{\ud-r}} \Sigma^{0}_{\ud-r}
   \]
   or equivalently
   \[
   \dim\cSigma^{r}_{\ud}=\dim\Sigma^{r}_{\ud}= \dim\Sigma^{0}_{\ud-r} + \left(2\left(\sum_{i=0}^N d_i\right)-d_0-d_N\right)r-N r^2.
   \]
\end{lemma}
\begin{proof}
We have
\[
\CR_{\ud}^r:=\{\ur+r\ |\ \ur\in \CR^0_{\ud-r}\}
\]
hence the result follows from Corollary~\ref{cor:irred_comp}, Theorem \ref{thm:addlongest} and Formula \eqref{eq:dim_Rep}.
\end{proof}

This also gives information about the geometry of the fibers of $\mult$, because of the following elementary observation.

\begin{lemma}\label{lem:rank_vs_fibers}
Let $\ud\in \N^{N+1}$ and $0\leq r\leq \min\ud$. Then (when $k=\C$) the restriction of $\mult$ to $\Sigma^r_{\ud}$ is a locally trivial bundle over the manifold $\Mat_{d_N,d_0}^{\rk=r}$. 

Hence for every $k$ and every $B\in \Mat_{d_N,d_0}^{\rk=r}$, there is a bijection between the irreducible components of $\Sigma^r_{\ud}$ and of $\mult^{-1}(B)$, and we have
\[
 \codim_{\Rep_{\ud}} \mult^{-1}(B)= \codim_{\Rep_{\ud}}  \Sigma^r_{\ud} + r(d_0+d_N-r)
\]
or equivalently
\[
\dim \mult^{-1}(B)=\dim \Sigma^r_{\ud} - r(d_0+d_N-r).
\]
\end{lemma}
\begin{proof}
We are free to assume $k=\C$ by Theorem \ref{thm:base_field}. As mentioned above, $\Mat_{d_N,d_0}^{\rk=r}$ is an orbit of the action of $G_\out$ on $\Mat_{d_N,d_0}$. Fix $B\in\Mat_{d_N,d_0}^{\rk=r}$ and $A_*\in \mult^{-1}(B)$ as base points. Let $G^B_\out$ be the stabilizer of $B$, so that the action map induces an isomorphism $\Mat_{d_N,d_0}^{\rk=r}$. Then local sections of the induced smooth submersion $G_\out\to \Mat_{d_N,d_0}^{\rk=r}$ induce local trivialisations of $\mult:\Sigma^r_{\ud}\to \Mat_{d_N,d_0}^{\rk=r}$ by $G_\out$-equivariance.
\end{proof}

\begin{remark}\label{rmk:real_points} \rm
    When $k=\R$, one can also ask about the irreducible components of $\cSigma^r_{\ud}$ as a \emph{real analytic set}. Those are not in general the same as the algebraically irreducible components, because the real points of the orbits $\CO_{\um}$ can be disconnected. So in particular one should be careful about how to interpret the number $\theta$ over $\R$. However by Theorem~\ref{thm:base_field}, we know that all algebraically irreducible components have real points, and that those real points have the same codimension as the algebraic codimension, so the number $C$ is unambiguous.

    Note that the connected components of $\mult^{-1}(B)(\R)$ are known \cite[Theorem 5]{trager-kohn-bruna:DLN-spurious-critical}.
\end{remark}

\begin{cor}\label{cor:min_ud}
 Let $\ud\in \N^{N+1}$ and $B\in \Mat_{d_N,d_0}^{\rk=\min{\ud}}$. Then $\Sigma^{\min{\ud}}_{\ud}$ and $\mult^{-1}(B)$ are smooth and (algebraically) connected varieties. 
\end{cor}
\begin{proof}
 This follows from Lemma~\ref{lem:rank_0} and Lemma~\ref{lem:rank_vs_fibers} together with the observation that, when $\min{\ud}=0$, we have $\Sigma^0_{\ud}=\Rep_{\ud}$.
\end{proof}

\section{Poincar\'e series} \label{sec:PoincareSection}
In this section we present a Poincar\'e series calculation that results in calculating the dimension and the number of top-dimensional components of $\Sigma_{\ud}^r$. We will use freely some tools from algebraic topology, namely equivariant cohomology of topological spaces acted upon by Lie groups. Those are standard tools in the study of quiver representations \cite{Reineke,RRcoha}.
\subsection{Pochammer symbols and the main statement}

\begin{defn}
    For a non-negative integer $s$ define
    \[
    \CP_s=\frac{1}{(1-q)(1-q^2)\cdots (1-q^s)},\]
    the inverse of the $q$-Pochhammer symbol. We will identify this function with its $|q|<1$ formal power series. For a multiset $\underline{h}=\{h_1,h_2,\ldots\}$ of nonnegative integers (i.e., a dimension vector or a Kostant partition), define 
    \[
    \CP_{\underline{h}}=\prod_i \CP_{h_i}.
    \]
\end{defn}

Given an evenly graded vector space $V=\oplus_n V_{2n}$, its Poincar\'e series is the formal power series $P(V)=\sum_n \dim(V_{2n}) q^n$. The inverse Pochammer symbol has an interpretation the generating series of the number of partitions with at most $s$ parts, and consequenctly as a Poincar\'e series in equivariant cohomology:
\[
\CP_s=\sum_{\mu_1\geq\ldots\geq \mu_s\geq 0} q^{|\mu|}=
P(\C[c_1,\ldots,c_s])=
P( H^*(B\!\GL_s(\C))) = P( H^*_{\GL_s}(\pt)).
\]
Here we have $\deg(c_i)=2i$. 

For later purposes we recall the infinite series expansion of the infinite Pochhammer symbol $(x;q)_{\infty}=\prod_{s=1}^\infty (1-xq^s)$ and its inverse 
\cite[Cor.~10.2.2]{special}: 
\begin{eqnarray} \notag
    \frac{1}{(x;q)_{\infty}}  &=& \sum_{s=0}^\infty \CP_s x^s \\
    \label{eqn:Ps_inverse}
    (x;q)_{\infty}  &=& \sum_{s=0}^\infty (-1)^s q^{\binom{s}{2}}\CP_s x^s.
\end{eqnarray}

We encode the codimensions of the orbits into a power series as follows:

\begin{defn}
Let $\ud\in \N^{N+1}$ and $r\in \N$.
\[Q_{\ud}^r:=\sum_{\substack{ \um \vdash \ud \\ m_{0N}=r}} q^{\codim(\CO_{\um})} \CP_{\um}\]
\end{defn}

The role of $Q_{\ud}^r$ in our study is clarified by:

\begin{lemma}\label{lem:Qs_codim}
     Let $0\leq r\leq\min\ud$, let $C$ be the codimension of $\cSigma^r_{\ud} \subset \Rep_{\ud}$ and $\theta$ the number of its top-dimensional irreducible components. Then  
    \[Q^r_{\ud}=\theta q^C+\text{higher order terms}.\]
\end{lemma}
\begin{proof}
       Let the lowest degree term of $Q^r_{\ud}$ be $nq^c$ (temporary notation). Since $\CP_{\um}=1+$higher degree terms, the definition of $Q^r_{\ud}$ implies that $c$ is the codimension of 
    $\bigcup_{\substack{ \um \vdash \ud \\ m_{0N}=r}} \CO_{\um}$. As we saw in Section \ref{sec:mult}, this union is precisely $\Sigma^r_{\ud}$, and we have moreover $\codim \Sigma^r_{\ud}=\codim \cSigma^r_{\ud}$. Moreover, $n$ is the number of orbits whose codimension is minimal. However, if an orbit closure is not a component, then it cannot have minimal codimension among orbit closures. Therefore, $n$ is also the number of the minimal codimensional {\em irreducible components}.
\end{proof}

\begin{remark} \rm
In the theory of characteristic classes of singularities---not discussed in this paper except in Appendix~\ref{sec:app_theta}---the sum $\sum_{s=0}^r Q^s_{\ud}$ is the Poincar\'e series of the ideal of cohomology classes ``universally supported'' on $\cSigma_{\ud}^r$.
\end{remark}

The main result of this section is then a purely algebraic formula for $Q^r_{\ud}$. 

\begin{thm} \label{thm:Pseries}
     We have
  \begin{equation}\label{eq:Fd_formula}
    Q^r_{\ud}=
    \CP_r\cdot \sum_{s=0}^{\min \ud-r} (-1)^s q^{\binom{s}{2}} \CP_s \CP_{\ud-r-s}.    
    \end{equation}
\end{thm}

We will prove Theorem~\ref{thm:Pseries} in Section~\ref{sec:Pseries_proof}.
Combined with Lemma~\ref{lem:Qs_codim}, this Theorem provides formulas for $C$ and $\theta$ which are not easy to write in closed form but can be implemented efficiently in a computer algebra system. 

\begin{example} \rm \label{ex:q-series}
The series 
\[
Q^0_{(2,2,2)}=q^3+6q^4+\ldots,
\qquad
Q^0_{(2,3,2)}=2q^4+7q^5+\ldots,
\qquad
Q^0_{(2,4,2)}=q^4+4q^5+\ldots
\]
correspond to the special cases discussed in Example \ref{ex:222324}.
We have
\[
Q^0_{(3,3,3)}=2q^7+8q^8+27q^9+67q^{10}+151q^{11}+\ldots,
\]
and hence $\Sigma_{(3,3,3)}^0$ has 2 top-dimensional components, of codimension $7$ (and hence, dimension $18-7=11$). We have
\[
Q^0_{(4,4,4,5,5,5,5,5,5,6,6,6)}=28q^{12}+508q^{13}+5129q^{14}+37424q^{15}+\ldots,
\]
and hence $\Sigma_{(4,4,4,5,5,5,5,5,5,6,6,6)}^0$ has 28 top-dimensional components, of codimension $12$. The statement is the same for any permutation of $(4,4,4,5,5,5,5,5,5,6,6,6)$, cf. Section~\ref{sec:perm}.
\end{example}

\subsection{The spectral sequence} 
\label{sec:proof_5gon}
The essence of the proof of Theorem~\ref{thm:Pseries} in Section~\ref{sec:Pseries_proof} is the understanding of a spectral sequence. We introduce this spectral sequence and the relevant arguments in the following known easier setting. 

\begin{thm} \label{thm:PdPm}
    We have
    \begin{equation}\label{eqn:5gon}
\CP_{\ud}
=
\sum_{\um \vdash \ud} 
q^{\codim(\CO_{\um})} 
\CP_{\um}.
\end{equation}
\end{thm}

This innocent-looking result is a special case of general combinatorial identities in the Donaldson-Thomas theory of arbitrary finite quivers \cite[Theorem 2.1]{Reineke}, \cite{RRcoha, RWY}. For $N=1$ (i.e. $Q=(\bullet \to \bullet)$), it is equivalent to the well-known {\em Durfee square identity}, or, the {\em pentagon identity of quantum dilogarithms}. 



\begin{proof}[Proof of Theorem~\ref{thm:PdPm}, following  \cite{RRcoha}] We saw in Section~\ref{sec:quiver_alg} that the group $G_{\ud}=$ $\prod_{i=0}^N\GL_{d_i}(\C)$ acts on $\Rep_{\ud}(\C)$ with finitely many orbits $\CO_{\um}$ indexed by Kostant partitions. Let $F_i$ denote the union of orbits with codimension at most $i$. Let us apply the Borel construction ($B_G X=BG \times_G X$) to the filtration
\[
F_0 \subset F_1 \subset F_2 \subset \ldots \subset  F_{\dim(\Rep_{\ud})} =\Rep_{\ud}(\C).
\]
We will study the ($\Z$-coefficient) cohomology spectral sequence associated with this filtration \cite{AB83, Kaz97, RRcoha}.

Stabilizer subgroups of different points in the same orbit $\CO_{\um}$ are conjugate, hence isomorphic. It is known (\cite[Prop~3.6]{FRduke}) that these stabilizers are homotopy equivalent to $G_{\um}=\prod_{ij} \GL_{m_{ij}}(\C)$.

\begin{lemma}\cite{Kaz97, RRcoha}
\label{lemma:SS}
    The spectral sequence $E_*^{pq}$ degenerates at $E_1$, it converges to $H^*(BG_{\ud})$, and 
\[
E_1^{pq}= \bigoplus_{\codim_{\Rep,\R} \CO_{\um}=p}
H^q(B G_{\um} ). 
\]
(where $\codim_{\R}=2\codim_{\C}$ is the real codimension).
\end{lemma}
\begin{proof}[Proof of the Lemma]
  The convergence claim follows from the fact that $\Rep_{\ud}$ is (equivariantly) contractible. The formula for $E_1^{pq}$ follows from the usual description of the $E_1$ page as relative cohomologies, if one applies excision and the Thom isomorphism to these relative cohomologies, see details in \cite{Kaz97, RRcoha}. The degeneration claim follows from the fact that all orbits have even real codimension and that $B G_{\um}$ 
has no odd cohomology---hence all differentials of the spectral sequence are 0.  
\end{proof}

Lemma~\ref{lemma:SS} implies that  
\[
\rk \left( H^{2n}(B G_{\ud}) \right) = \sum_{p+q=2n} \rk(E_1^{pq}). 
\]
Multiplying this identity with $q^n$ and summing for all $n$---using Lemma~ \ref{lemma:SS}---we obtain 
\begin{equation} \label{eq:gyor}
    P(H^*(B G_{\ud}))=\sum_{\um \vdash \ud} q^{\codim(\CO_{\um})} P(H^*(B G_{\um})),
\end{equation}
and the statement of Theorem~\ref{thm:PdPm} follows.
\end{proof}

\subsection{Proof of Theorem \ref{thm:Pseries}}
\label{sec:Pseries_proof}

\begin{lemma}\label{lem:shift}
For $0\leq r\leq s$ we have
\[
Q_{\ud-r}^{s-r}= (1-q^{s-r+1})(1-q^{s-r+2})\cdots (1-q^{r}) \cdot Q_{\ud}^{s}.
\]
\end{lemma}

\begin{proof}
There is a bijection between Kostant partitions $\um$ of $\ud-r$ with $m_{0N}=s-r$ and Kostant partitions $\um'$ of $\ud$ with $m_{0N}=s$. The bijection is just adding $r$ to the component $m_{0N}$, cf. Section \ref{sec:add_longest}. We obtain
\begin{equation}\label{eq:temp}
Q_{\ud-r}^{s-r}=
\sum_{\substack{ \um \vdash \ud-r \\ m_{0N}=s-r}} q^{\codim(\CO_{\um})} \CP_{\um} 
=
\sum_{\substack{ \um' \vdash \ud \\ m'_{0N}=s}} q^{\codim(\CO_{\um'})} \CP_{\um'}. 
\end{equation}
According to Theorem \ref{thm:addlongest} (based on Lemma \ref{lem:voight}) we have $\codim(\CO_{\um})=\codim(\CO_{\um'})$, and 
\[
\CP_{\um}=
\CP_{\um'} \cdot (1-q^{s-r+1})(1-q^{s-r+2})\cdots(1-q^r).
\]
Hence, \eqref{eq:temp} is further equal
\begin{multline*}
(1-q^{s-r+1})(1-q^{s-r+2})\cdots(1-q^r) 
\sum_{\substack{ \um' \vdash \ud \\ m'_{0N}=s}} q^{\codim(\CO_{\um'})} \CP_{\um'}
\\=
(1-q^{s-r+1})(1-q^{s-r+2})\cdots(1-q^r) Q_{\ud}^{s},
\end{multline*}
as we wanted to prove.
\end{proof}

Observe that if we prove Theorem \ref{thm:Pseries} for $r=0$, then the $r>0$ cases follow using Lemma~\ref{lem:shift}. So, from now on we focus on the $r=0$ case.

\smallskip

We now use Theorem~\ref{thm:PdPm}, which by definition of $Q^s_{\ud}$ can be rewritten as
\begin{equation}\label{PdQs}
\CP_{\ud}=\sum_{s=0}^{\min \ud} Q_{\ud}^s.
\end{equation}

Apply Lemma~\ref{lem:shift} for each term of the right hand side with $r=s$. We obtain
\begin{equation}\label{eqn:key}
\CP_{\ud}=\sum_{s=0}^{\min \ud}
\CP_s Q_{\ud-s}^{0}.
\end{equation}
Using the temporary notation
\[
d'_0=\min \ud,
\qquad
a_i=\CP_{\ud-d'_0+i}
\qquad
b_i=Q_{\ud-d'_0+i}^{0}
\]
equations \eqref{eqn:key} for $\ud, \ud-1, \ud-2, \ldots$ together are equivalent to the single equation
\[
\left(
\sum_{s=0}^\infty b_s x^s 
\right)
\left( 
\sum_{s=0}^\infty \CP_s x^s 
\right)
=
\sum_{s=0}^\infty a_s x^s,
\]
of power series in a new formal variable $x$. From this, using \eqref{eqn:Ps_inverse} we obtain
\[
\sum_{s=0}^\infty b_s x^s 
=
\left(
\sum_{s=0}^\infty a_s x^s 
\right)
\left( 
\sum_{s=0}^\infty \CP_s x^s 
\right)^{-1}=
\left(
\sum_{s=0}^\infty a_s x^s 
\right)
\left( 
\sum_{s=0}^\infty (-1)^s q^{\binom{s}{2}}\CP_s x^s 
\right).
\]
For the coefficient of $x^{d'_0}$ this yields
\[
b_{d'_0}=\sum_{s=0}^{d'_0} a_{r-s} (-1)^s q^{\binom{s}{2}} \CP_s,
\]
which---unfolding our temporary definitions---is the statement of Theorem \ref{thm:Pseries} for $r=0$.  \qed

\subsection{Permutation invariance} \label{sec:perm}

Theorem~\ref{thm:Pseries} has the following consequence, which we found quite surprising (cf. Remark \ref{remark:surprise}). 

\begin{cor} \label{cor:PermutationInvariance}
    The numbers $C$, $\theta$ are invariant under permutations of~$\ud$.
\end{cor}

\begin{proof}
    All ingredients on the right hand side of \eqref{eq:Fd_formula} are invariant under such permutations.
\end{proof}

In Appendix~\ref{sec:app_theta} we will give another proof of this invariance. We do not know of a proof relying solely on equation \eqref{eqn:dim_formula} and the combinatorics of Kostant partitions. Note that given a permutation $\sigma(\ud)$ of $\ud$, the sets of Kostant partitions $\CM^+_{\ud}$ and $\CM^+_{\sigma(\ud)}$ do not necessarily have the same cardinality, and in particular $\sigma$ does not induce a bijection between them.

\begin{remark} \rm \label{remark:surprise}
 One reason that Corollary \ref{cor:PermutationInvariance} is surprising is that the geometry of $\cSigma^r_{\ud}$ undergoes significant changes when the components of $\ud$ are permuted. First of all, it is clear from equation \eqref{eq:dim_Rep} that the \emph{dimension} of $\Rep_{\ud}$ is not permutation invariant, so that it is only the codimension of $\cSigma_{\ud}^r$ which is invariant and not its dimension. For example,
 \[
 \dim(\Rep_{(2,2,3)})=10\neq 12=\dim(\Rep_{(2,3,2)}).
 \]
Moreover, the \emph{total number} of irreducible components of $\cSigma_{\ud}^r$ is not permutation invariant. Using Corollary \ref{cor:irred_comp}, one can easily check that $\Sigma_{(2,3,2)}^0$ has 2 irreducible components (both of codimension $4$) while $\Sigma_{(2,2,3)}^0$ and $\Sigma_{(3,2,2)}^0$ have 3 irreducible components (2 of codimension 4 and 1 of codimension 6).
\end{remark}

\section{Reducing the problem to a quadratic integer program}

While Theorem \ref{thm:Pseries} is a very effective way of calculating the number of top-dimensional components of $\Sigma^0_{\ud}$, it is useless for asymptotic analysis. Hence, it is desirable to have more explicit formulas. Towards such, we first reduce the problem of calculating $\codimnot$ and $\theta$ to a quadratic integer program. The argument is {\em almost} independent of Section~\ref{sec:PoincareSection}; almost, because, in the case when the dimension vector $\ud$ is not weakly increasing, we use Corollary~\ref{cor:PermutationInvariance}. As we saw above, it is enough to treat the case $r=0$ and we do so in this section.

\subsection{Top-dimensional components of $\Sigma^0_{\ud}$ and a QIP}
For a dimension vector $\ud=(d_0,d_1,\ldots,d_N)$ let $\ud'=(d'_0\leq d'_1\leq \ldots \leq d'_N)$ be the same multiset, but arranged (weakly) increasingly. Consider $N$-tuples of non-negative integers $\ue=(e_1,e_2,\ldots,e_N)$ and the following quadratic integer program for them:

\begin{align}
\tag{QIP} \label{eq:QIP}
& \begin{array}{rl} \min & \displaystyle G_{\ud}(\ue)=\sum_{1\leq j\leq i\leq N} e_i(e_j+d'_j-d'_{j-1}) \\[.1in] \text{s.t.} & \displaystyle \ e_i\in \N, \quad \sum_{i=1}^N e_i=d'_0.
\end{array}
\end{align}

\begin{thm} \label{thm:QIP}
    The codimension of the top-dimensional components of $\Sigma^0_{\ud}$ is the minimal value of \eqref{eq:QIP}. The number of top-dimensional components of $\Sigma^0_{\ud}$ is the number of times \eqref{eq:QIP} attains its minimum value.
\end{thm}

\begin{example} \rm \label{ex:223withQIP}
  For $\ud=(2,2,3)$ the \eqref{eq:QIP} is minimizing the function $G_{\ud}(\ue)=e_1^2+e_1e_2+e_2^2+e_2$ on the three-element set $\{(0,2),(1,1),(2,0)\}$. The minimum is 4, attained at the second and third element of the set. Hence $\theta=2$, $C=4$, cf. Examples \ref{ex:222324}, \ref{ex:q-series}.  
\end{example}

\begin{example} \rm \label{ex:N8}
  For $\ud=(8,8,11,11,11,13,13,13,15)$ the minimal value of \eqref{eq:QIP} is 55, and it is achieved at four points: $e=(4,2,1,1,0,0,0,0)$, $e=(4,1,2,1,0,0,0,0)$, $e=(4,1,1,2,0,0,0,0)$, $e=(5,1,1,1,0,0,0,0)$. 
\end{example}

\subsection{Proof of Theorem \ref{thm:QIP}} 
\label{sec:ProofOfQIP}

In view of Corollary \ref{cor:PermutationInvariance} it is sufficient to prove Theorem~\ref{thm:QIP} for the dimension vector $\ud'$ whose components are weakly increasing.

We need to establish a useful fact about lace diagrams in the weakly increasing case. Recall that a lace diagram is an arrangements of dots in columns $0,1,\ldots,N$, partitioned to a number of $[i,j]$ intervals (laces) and that Kostant partitions are equivalence classes of lace diagrams under permutations (Lemma \ref{lem:Kostant_lace}).

Given a Kostant partition, it is not always possible to choose a lace diagram with only {\em horizontal laces}. For example, the lace diagrams attached to a certain Kostant partition in $\CM^+_{(1,2,1)}$ are

\begin{equation} \label{eqn:smalllaces}
\begin{tikzpicture}[baseline=-5]
 \node at (0,0)  {$\bullet$};
 \node at (1,0)  {$\bullet$};
 \node at (1,-.2)  {$\bullet$};
 \node at (2,0)  {$\bullet$};
  \draw (0,0) -- (1,-.2);
  \draw (1,0)--(2,0);
\end{tikzpicture}
\quad\quad
\begin{tikzpicture}[baseline=-5]
 \node at (0,0)  {$\bullet$};
 \node at (1,0)  {$\bullet$};
 \node at (1,-.2)  {$\bullet$};
 \node at (2,0)  {$\bullet$};
  \draw (0,0) -- (1,0);
  \draw (1,-.2)--(2,0);
\end{tikzpicture}.
\end{equation}

\begin{lemma} \label{lem:horiz_rep}
Assume that the dimension vector $\ud$ is weakly increasing. Then for any Kostant partition $\um\in \CM^+_{\ud}$ there is a lace diagram $L$ with $\um(L)=\um$ and only horizontal laces.
\end{lemma}

\begin{proof} Induction on $N$ (the number of columns).\end{proof}


Let $L$ be a lace diagram representing the orbit $\CO_{\um}$ of $\Rep_{\ud'}$ whose closure is a top-dimensional component of $\Sigma^0_{\ud'}$. We choose $L$ to have only horizontal laces by Lemma~\ref{lem:horiz_rep}.

For our next arguments the reader is advised to look at Figure~\ref{fig:largeLD}. We will call the horizontal lines of dot positions in $L$ ``rows''. Since the orbit $\CO_{\um}$ is in $\Sigma^0_{\ud'}$, none of the rows in $L$ can be the full $[0,N]$ interval. The top $d'_0$ rows have dots in all columns, hence we obtain that in the top $d'_0$ rows {\em at least one of the $[i-1,i]$ intervals} is not part of the lace diagram. 

We claim that in all these rows {\em exactly one $[i-1,i]$ interval} is missing and all the others are laces. Assume to the contrary that there are two such intervals missing in a row. Adding one of them back, the orbit of this new lace diagram would contain $\CO_{\um}$ in its closure, and it would still belong to $\Sigma^0_{\ud'}$. This contradicts the fact that $\overline{\CO_{\um}}$ is a  component of $\Sigma^0_{\ud'}$.

An analoguous argument shows that in the rows below the top $d'_0$ rows all possible horizontal intervals are laces of $L$. In effect $L$ has the structure indicated by the example in Figure~\ref{fig:largeLD}.

What we found about the structure of $L$ translates to algebra as follows: the $Q$-module corresponding to $L$ is of the form
\[
\begin{array}{lrclcl}
M= &      e_1(I_{00}+I_{1N})  & +& e_2(I_{01}+I_{2N}) & + ~ \ldots ~ + & e_N(I_{0,N-1}+I_{NN})  \\
&
   + ~ f_1({\hskip 0.9 true cm} I_{1N}) 
      & +& f_2({\hskip .91 true cm} I_{2N}) 
        & + ~ \ldots ~ + 
           & f_N({\hskip 1.45 true cm}I_{NN})
\end{array}
\]
for some integers $e_i$ and $f_i$, see again Figure~\ref{fig:largeLD}. In fact $f_i$ are determined by the dimension vector $\ud'$ as $f_j=d'_j-d'_{j-1}$, while $e_i$ can be any non-negative integers whose sum is $d'_0$. 
The codimension of $\CO_{\ud}$ is $\dim\Ext(M,M)$ (see Lemma \ref{lem:voight}). Using the bilinearity of $\Ext$ and \eqref{eqn:ExtIndecomposables} we obtain
\begin{multline*}
\dim \Ext(M,M)= \sum_{1\leq j\leq i\leq N} e_i(e_j+f_j) \dim \Ext(I_{0,{i-1}},I_{jN}) 
\\=
\sum_{1\leq j\leq i\leq N} e_i(e_j+f_j)
=\sum_{1\leq j\leq i\leq N} e_i(e_j+d'_j-d'_{j-1})=G_{\ud}(\ue),
\end{multline*}
and the theorem follows. \qed

\begin{remark} \rm
  When $\ud$ is weakly increasing, the proof of Theorem \ref{thm:QIP} provides an explicit bijection between the set of top-dimensional components of $\Sigma^0_{\ud}$ and the solutions of \eqref{eq:QIP}. 
  The combinatorial description of the $\theta$ top dimensional components of $\Sigma^0_{\ud}$ for arbitrary $\ud$ (not necessarily weakly increasing) will be given in \cite{KR}.
\end{remark}

\begin{figure}
\begin{equation*}
\label{eq:lace_diagram2}
\begin{tikzpicture}[baseline=-12]
\node at (0,.5) {$0$};
\node at (1,.5) {$1$};
\node at (2,.5) {$2$};
\node at (3,.5) {$3$};
\node at (4,.5) {$4$};
\node at (5,.5) {$5$};
\node at (6,.5) {$6$};
\node at (7,.5) {$7$};
 
\draw[blue] (1,0) -- (7,0); \draw[blue] (1,-.2) -- (7,-.2);
\draw[blue] (0,-.4) -- (1,-.4); \draw[blue] (2,-.4) -- (7,-.4);
\draw[blue] (0,-.6) -- (1,-.6); \draw[blue] (2,-.6) -- (7,-.6);
\draw[blue] (0,-.8) -- (1,-.8); \draw[blue] (2,-.8) -- (7,-.8);
\draw[blue] (0,-1) -- (4,-1); \draw[blue] (5,-1) -- (7,-1);
\draw[blue] (0,-1.2) -- (4,-1.2); \draw[blue] (5,-1.2) -- (7,-1.2);
\draw[blue] (0,-1.4) -- (6,-1.4); 
  \draw[red] (2,-1.6) -- (7,-1.6);
  \draw[red] (2,-1.8) -- (7,-1.8);
  \draw[red] (2,-2) -- (7,-2);
  \draw[red] (4,-2.2) -- (7,-2.2);
  \draw[red] (4,-2.4) -- (7,-2.4);

\draw[blue] (7.5,0.05) to[out=0,in=180] (7.6,-0.1);
\draw[blue] (7.5,-0.25) to[out=0,in=180] (7.6,-0.1);
\node[blue] at (8.3,-0.1)  {$e_1=2$}; 

\draw[blue] (7.5,-.35) to[out=0,in=180] (7.6,-0.6);
\draw[blue] (7.5,-0.85) to[out=0,in=180] (7.6,-0.6);
\node[blue] at (8.3,-0.6)  {$e_2=3$}; 

\draw[blue] (7.5,-.95) to[out=0,in=180] (7.6,-1.1);
\draw[blue] (7.5,-1.25) to[out=0,in=180] (7.6,-1.1);
\node[blue] at (8.3,-1.1)  {$e_5=2$}; 

\draw[blue] (7.5,-1.31) to[out=0,in=180] (7.56,-1.4);
\draw[blue] (7.5,-1.49) to[out=0,in=180] (7.56,-1.4);
\node[blue] at (8.3,-1.4)  {$e_7=1$}; 

\draw[red] (7.5,-1.55) to[out=0,in=180] (7.6,-1.8);
\draw[red] (7.5,-2.05) to[out=0,in=180] (7.6,-1.8);
\node[red] at (8.3,-1.8)  {$f_2=3$}; 

\draw[blue] (7.5,-.95) to[out=0,in=180] (7.6,-1.1);
\draw[blue] (7.5,-1.25) to[out=0,in=180] (7.6,-1.1);
\node[blue] at (8.3,-1.1)  {$e_5=2$}; 

\draw[red] (7.5,-2.15) to[out=0,in=180] (7.6,-2.3);
\draw[red] (7.5,-2.45) to[out=0,in=180] (7.6,-2.3);
\node[red] at (8.3,-2.3)  {$f_4=2$}; 

\draw[red] (7.5,-2.55) to[out=0,in=180] (7.6,-2.7);
\draw[red] (7.5,-2.85) to[out=0,in=180] (7.6,-2.7);
\node[red] at (8.3,-2.7)  {$f_7=2$}; 

\draw (-.4,.05)  to[out=180,in=0]  (-.7,-.7);
\draw (-.4,-1.45) to[out=180,in=0] (-.7,-.7);
\node at (-1,-.7)  {$d'_0$}; 

\node[black] at (5.85,-3.2)  {($e_3=e_4=e_6=f_1=f_3=f_5=f_6=0$)}; 

\node at (0,0)  {$\bullet$};
 \node at (0,-.2)  {$\bullet$};
 \node at (0,-.4)  {$\bullet$};
 \node at (0,-.6)  {$\bullet$};
 \node at (0,-.8)  {$\bullet$};
 \node at (0,-1)  {$\bullet$};
 \node at (0,-1.2)  {$\bullet$};
 \node at (0,-1.4)  {$\bullet$};

 \node at (1,0)  {$\bullet$};
 \node at (1,-.2)  {$\bullet$};
 \node at (1,-.4)  {$\bullet$};
 \node at (1,-.6)  {$\bullet$};
 \node at (1,-.8)  {$\bullet$};
 \node at (1,-1)  {$\bullet$};
 \node at (1,-1.2)  {$\bullet$};
 \node at (1,-1.4)  {$\bullet$};

 \node at (2,0)  {$\bullet$};
 \node at (2,-.2)  {$\bullet$};
 \node at (2,-.4)  {$\bullet$};
 \node at (2,-.6)  {$\bullet$};
 \node at (2,-.8)  {$\bullet$}; 
 \node at (2,-1)  {$\bullet$};
 \node at (2,-1.2)  {$\bullet$};
 \node at (2,-1.4)  {$\bullet$};
 \node at (2,-1.6)  {$\bullet$};
 \node at (2,-1.8)  {$\bullet$};
 \node at (2,-2)  {$\bullet$};

 \node at (3,0)  {$\bullet$};
 \node at (3,-.2)  {$\bullet$};
 \node at (3,-.4)  {$\bullet$};
 \node at (3,-.6)  {$\bullet$};
 \node at (3,-.8)  {$\bullet$};
 \node at (3,-1)  {$\bullet$};
 \node at (3,-1.2)  {$\bullet$};
 \node at (3,-1.4)  {$\bullet$};
 \node at (3,-1.6)  {$\bullet$};
 \node at (3,-1.8)  {$\bullet$};
 \node at (3,-2)  {$\bullet$};

 \node at (4,0)  {$\bullet$};
 \node at (4,-.2)  {$\bullet$};
 \node at (4,-.4)  {$\bullet$};
 \node at (4,-.6)  {$\bullet$};
 \node at (4,-.8)  {$\bullet$};
 \node at (4,-1)  {$\bullet$};
 \node at (4,-1.2)  {$\bullet$};
 \node at (4,-1.4)  {$\bullet$};
 \node at (4,-1.6)  {$\bullet$};
 \node at (4,-1.8)  {$\bullet$};
 \node at (4,-2)  {$\bullet$};
 \node at (4,-2.2)  {$\bullet$};
 \node at (4,-2.4)  {$\bullet$};

 \node at (5,0)  {$\bullet$};
 \node at (5,-.2)  {$\bullet$};
 \node at (5,-.4)  {$\bullet$};
 \node at (5,-.6)  {$\bullet$};
 \node at (5,-.8)  {$\bullet$};
 \node at (5,-1)  {$\bullet$};
 \node at (5,-1.2)  {$\bullet$};
 \node at (5,-1.4)  {$\bullet$};
 \node at (5,-1.6)  {$\bullet$};
 \node at (5,-1.8)  {$\bullet$};
 \node at (5,-2)  {$\bullet$};
 \node at (5,-2.2)  {$\bullet$};
 \node at (5,-2.4)  {$\bullet$};

 \node at (6,0)  {$\bullet$};
 \node at (6,-.2)  {$\bullet$};
 \node at (6,-.4)  {$\bullet$};
 \node at (6,-.6)  {$\bullet$};
 \node at (6,-.8)  {$\bullet$};
 \node at (6,-1)  {$\bullet$};
 \node at (6,-1.2)  {$\bullet$};
 \node at (6,-1.4)  {$\bullet$};
 \node at (6,-1.6)  {$\bullet$};
 \node at (6,-1.8)  {$\bullet$};
 \node at (6,-2)  {$\bullet$};
 \node at (6,-2.2)  {$\bullet$};
 \node at (6,-2.4)  {$\bullet$};

 \node at (7,0)  {$\bullet$};
 \node at (7,-.2)  {$\bullet$};
 \node at (7,-.4)  {$\bullet$};
 \node at (7,-.6)  {$\bullet$};
 \node at (7,-.8)  {$\bullet$};
 \node at (7,-1)  {$\bullet$};
 \node at (7,-1.2)  {$\bullet$};
 \node at (7,-1.4)  {$\bullet$};
 \node at (7,-1.6)  {$\bullet$};
 \node at (7,-1.8)  {$\bullet$};
 \node at (7,-2)  {$\bullet$};
 \node at (7,-2.2)  {$\bullet$};
 \node at (7,-2.4)  {$\bullet$};
 \node at (7,-2.6)  {$\bullet$}; 
 \node at (7,-2.8)  {$\bullet$}; 

\end{tikzpicture}
\end{equation*}
\caption{A lace diagram with weakly increasing dimension vector and horizontal laces, cf. Section~\ref{sec:ProofOfQIP}.} \label{fig:largeLD}
\end{figure}
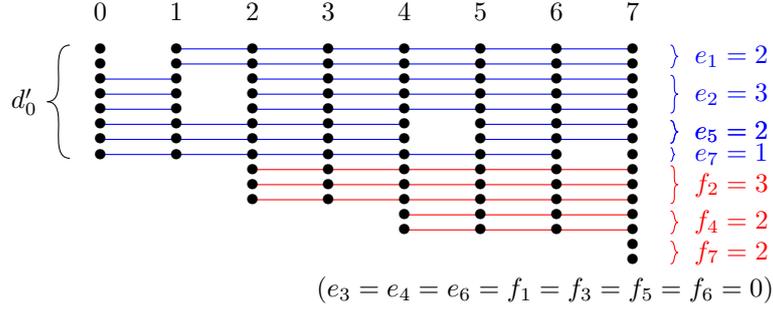

\section{Explicit formulas for the codimensions of loci} \label{sec:SolvingQIP}

In this section, we make use of the rather simple structure of \eqref{eq:QIP} and give an explicit solution to it. The final result is stated in Theorem \ref{thm:main-codim} below, but this requires some additional notation.

\begin{lemma}
    The function 
    \[l\in \{0,1,\ldots,N\}\mapsto \left(\sum_{i=0}^l d'_i\right) - l d'_l\]
    is weakly decreasing and positive for $l=1$.
\end{lemma}
\begin{proof}
    This follows from
    \[
\left(\sum_{i=0}^l d'_i\right) - (l+1) d'_{l+1} - \left(   \left(\sum_{i=0}^l d'_i\right) - l d'_l\right)=l(d'_l-d'_{l+1})\leq 0
    \]
since $\ud'$ is weakly decreasing. The value for $l=1$ is $d'_0>0$.
\end{proof}

Depending of $\ud$, this function may or may not become negative for $l=N$. We can, however, always make the following definition.

\begin{defn} \label{def:d1}
Define 
\[
m=\max \left\{l\ |\ \sum_{i=0}^l d'_i \geq l d'_l\right\} \in \{1,2,\ldots,N\}.
\]
\end{defn}

To solve \eqref{eq:QIP}, we first define an integer $D$ and vectors $\us, \hat{\us}$ associated with the weakly increasing dimension vector $\ud'=(d'_0\leq d'_1 \leq \ldots \leq d'_N)$. The reader may find Figure~\ref{Fig:distances} and Example~\ref{ex:closestpoint} useful to understand these definitions.

\begin{defn} \label{def:d2}
Consider the points
\[
\begin{array}{l}
    \us:=\left(d'_0-d'_1, d'_0-d'_2,d'_0-d'_3,\ldots\ldots\ldots,d'_0-d'_N \right) \in \Z^N,
    \\
    \hat{\us}:=\left(d'_0-d'_1, d'_0-d'_2,d'_0-d'_3,\ldots,d'_0-d'_m \right) \in \Z^m,
\end{array}
\]
and let $\hat{\uv}_1,\hat{\uv}_2,\ldots,\hat{\uv}_k$ be the set of integer points in the hyperplane $\{\sum_{i=1}^m x_i=d'_0\}$ closest to~$\hat{\us}$ in the Euclidean distance.
\end{defn}

\begin{defn}\label{def:d3}
    Define 
    \[\hat{D}:=\| \hat{\us}-\hat{\uv}_i \|^2\]
    (which is independent of $i$ by definition). Define $\uv_i$ to be the image of $\hat{\uv}_i$ under the standard embedding $\Z^m \subset \Z^N$, $(x_1,\ldots,x_m)\mapsto (x_1,\ldots,x_m,0,\ldots,0)$ and $D:=\| \us-\uv_i \|^2$.
\end{defn}

By construction we have
\begin{equation}\label{eqn:D_Dhat}
D=\hat{D}+\sum_{i=m+1}^N(d'_i-d'_0)^2.
\end{equation}


\begin{thm} \label{thm:QIPreduced2closestpoint}
    For $\ud'=(d'_0\leq d'_1 \leq \ldots\leq d'_N)$ the solution of the quadratic integer program \eqref{eq:QIP} (ie. the codimension of $\Sigma^0_{\ud}$) is  
     \begin{equation}\label{eq:gD}
    \frac{1}{2}\left( \hat{D} + {d'_0}^2 - \sum_{i=1}^m (d'_i-d'_0)^2  \right), 
    \end{equation}
    and this minimum is attained at $k$ points (ie. $\Sigma^0_{\ud}$ has $k$ largest dimensional components).
\end{thm}


\begin{proof}
Consider the \eqref{eq:QIP}. 
Denoting the objective function 
$\sum_{j\leq i} e_i(e_j+d'_j-d'_{j-1})$ by $G_{\ud}(\ue)$, we have
\begin{eqnarray*}
2G_{\ud}(\ue)-{d'_0}^2 &=& 2G_{\ud}(\ue) - \left( \sum e_i \right)^2\\
&=& \sum_{j\leq i} e_i(d'_j-d'_{j-1}) + \sum_i e_i^2\\
&=& \sum_i 2e_i(d'_i-d'_0) +e_i^2\\
&=& \left(\sum_i (e_i-(d'_0-d'_i))^2 \right)
-\sum_i (d'_i-d'_0)^2
\end{eqnarray*}

Therefore, using the notations from Definitions~\ref{def:d1}-\ref{def:d3}, for the optimal solution $G_{\ud}^{opt}$ of \eqref{eq:QIP} we get
\[
2G_{\ud}^{opt}-{d'_0}^2+\sum_{i=1}^N s_i^2 = D,
\]
where $D$ is the smallest distance-square between $\us$ and integer points of the simplex
\[ 
\sum_{i=1}^N e_i=d'_0, \qquad  e_i\geq 0,
\]
cf. the left picture in Figure~\ref{Fig:distances}. The 
projection of $\us$ to the hyperplane $\sum_{i=1}^N e_i=d'_0$ is 
\[
\us+ \left(\frac{d'_0 - \sum_{i=1}^N s_i}{N} \right)(1,1,\ldots,1).
\]
Because of our choices the components of this vector are weakly decreasing; and possibly the few last components are negative. In those coordinates the optimal choice for $e_i$ must be 0, and our task reduces to the analogous ``smallest distance-square'' problem in the smaller dimensional space with coordinates $e_1,\ldots,e_l$. We can drop the last few coordinates and keep only the first $l$ coordinates, until the vector 
\[
(s_1,s_2,\ldots,s_l) + \left(\frac{d'_0 - \sum_{i=1}^l s_i}{l}\right)(1,1,\ldots,1).
\]
has all non-negative coordinates. The condition for the latter is 
\[
s_l+\left( \frac{d'_0- \sum_{i=1}^l s_i}{l} \right) \geq 0.
\]
Substituting the definition $s_l=d'_0-d'_l$, this condition reduces to the relation in Definition~\ref{def:d1}: $\sum_{i=0}^l d'_i \geq ld'_l$. That is we can drop all but the first $m$ coordinates.

Together with equation \ref{eqn:D_Dhat}, this completes the proof of Theorem~\ref{thm:QIPreduced2closestpoint}, both about the extreme value of $\eqref{eq:QIP}$ and about the number of times the extreme value is attained.
\end{proof}

\begin{figure}
\begin{equation*}
\begin{tikzpicture}[scale=.8]
\draw[->] (0,0) -- (4,-.8); 
\draw[->] (0,0) -- (4,1.5);
\draw[->] (0,0) -- (0,5);
\draw[thick] (2.5,-.5) -- (0,4) -- (3,1.5*3/4) -- (2.5,-.5);
\node at (2.5,-.5) {$\bullet$};
\node[blue] at (2.3,-.5+.2*4.5/2.5) {$\bullet$};
\node[blue] at (2.1,-.5+.4*4.5/2.5) {$\bullet$};
\node at (1.9,-.5+.6*4.5/2.5) {$\bullet$};
\node at (1.7,-.5+.8*4.5/2.5) {$\bullet$};
\node at (2.6,-.5+13/40+ 0*4.5/2.5) {$\bullet$};
\node at (2.4,-.5+13/40+.2*4.5/2.5) {$\bullet$};
\node at (2.2,-.5+13/40+.4*4.5/2.5) {$\bullet$};
\node at (2.0,-.5+13/40+.6*4.5/2.5) {$\bullet$};
\node at (2.7,-.5+2*13/40+ 0*4.5/2.5) {$\bullet$};
\node at (2.5,-.5+2*13/40+.2*4.5/2.5) {$\bullet$};
\node at (0,4) {$\bullet$};
\node at (3,9/8) {$\bullet$};
\draw[blue] (-2.5,-2) -- (2.3,-.5+.2*4.5/2.5);
\draw[blue] (-2.5,-2) -- (2.1,-.5+.4*4.5/2.5);
\node[blue] at (-2.5,-2) {$\bullet$};
\node[blue] at (-2.6,-1.5) {$\us$};
\node at (4,-.5) {$e_1$};
\node at (.4,4.7) {$e_2$};
\node at (4,1.7) {$e_3$};
\node at (2.7,3.1) {$\sum^N\!e_i=d'_0$};
\node at (2.7,2.5) {$e_i\geq 0$};
\node at (4,4.5) {$\R^N$};
\node at (-2,1.5) {distance${}^2=D$};
\draw[dashed,->] (-1.8,1.2) -- (-.4,-.95);
\draw[dashed,->] (-1.8,1.2) -- (-.5,-1.2);
\node[blue] at (2,-2) {$\uv_1,\uv_2$};
\draw[dashed,blue,->] (1.45,-1.7) to[out=120,in=-120] (2.05,.1);
\draw[dashed,blue,->] (2.1,-1.7) to[out=120,in=-120] (2.25,-.25);
\end{tikzpicture}
\qquad\qquad\qquad
\begin{tikzpicture}[scale=.8,baseline=-50]
\draw[->] (0,0) -- (4,0); 
\draw[->] (0,0) -- (0,4);
\draw[thick] (4,-1)--(-1,4);
\node at (3.3,-.3) {$\bullet$};
\node at (3.6,-.6) {$\bullet$};
\node at (3,0) {$\bullet$};
\node[blue] at (2.7,.3) {$\bullet$};
\node[blue] at (2.4,.6) {$\bullet$};
\node at (2.1,.9) {$\bullet$};
\node at (0,3) {$\bullet$};
\node at (-.3,3.3) {$\bullet$};
\node at (0.3,2.7) {$\bullet$};
\draw[blue] (2.55-3,.3-3) -- (2.7,.3);
\draw[blue] (2.55-3,.3-3) -- (2.4,.6);
\node at (4.4,-.15) {$e_1$};
\node at (.45,3.8) {$e_2$};
\node at (2.3,2.2) {$\sum^m\!e_i=d'_0$};
\node[blue] at (2.55-3,.3-3) {$\bullet$};
\node[blue] at (2.55-3-.3,.3-3+.5) {$\hat\us$};
\node[blue] at (3.15,0.55) {$\hat{\uv}_2$};
\node[blue] at (2.7,0.99) {$\hat{\uv}_1$};
\node at (4,3.5) {$\R^m$};
\node[blue] at (4.6,1.3) {$k=2$};
\draw[dashed,->] (4,1.3) -- (3,1.05);
\draw[dashed,->] (4,1.3) -- (3.3,.7);
\draw[dashed,->,red] (3,-2) to [out=160,in=-130] (2.5,.4);
\node[red] at (3.2,-2) {$\hat{\up}$};
\end{tikzpicture}
\end{equation*}
\caption{Illustration of the concepts in Section~\ref{sec:SolvingQIP}.} \label{Fig:distances}
\end{figure}
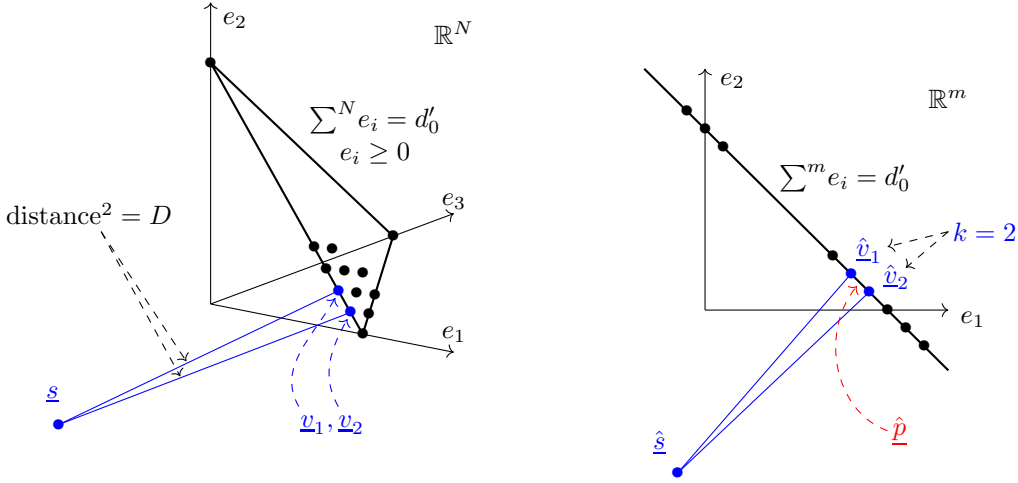

To get more explicit formulas from Theorem \ref{thm:QIPreduced2closestpoint}, it remains to find the integer vectors in the $\{\sum_{i=1}^m e_i=d'_0\}$ hyperplane closest to a given vector $\hat{\us}$ in $\Z^m$. Luckily, this problem has an explicit solution \cite[Ch.~20]{ConwaySloane}, that we describe now, for our special case.

Consider 
\[
\hat{\up} =\hat{\us} + \left( \frac{d'_0-\sum_{j=1}^m \hat{\us}_j}{m} \right)(1,1,\ldots,1)\quad \in \R^m,
\]
this is the projection of $\hat{\us}$ to the hyperplane $\sum_{i=1}^m\! e_i=d'_0$. Writing
\[
S:= \sum_{j=0}^m d'_j
\]
we have for every $1\leq i\leq m$
\[
\hat{\up}_i =(d'_0-d'_i)+\frac{d'_0-\left(\sum_{j=1}^m(d'_0-d'_j)\right)}{m} =d'_i-\frac{S}{m}.
\]
Rounding it to the closest integer coordinate-wise will be denoted by $r(\hat{\up})=(r(\hat{p}_i))_{i=1,\ldots,m}$, 
where $r(x)=\lfloor x+\frac12 \rfloor$, so that
\[
r(\hat{\up}_i)=-d_i+ \left\lfloor \frac{S_0}{m}+\frac{1}{2}\right\rfloor.
\]
This vector may not be in the hyperplane and we need to correct for this. Let 
\begin{equation}\label{eqn:delta}
    \delta=d'_0-\sum_{i=1}^m r(\hat{p}_i)=S-m\left\lfloor \frac{S}{m}+\frac{1}{2}\right\rfloor
\end{equation}
and let $\epsilon=\mathrm{sgn}(\delta) \in \{-1,0,1\}$ be its sign. Let
$\Delta_1,\Delta_2,\ldots \in \{-1,0,1\}^m$ be the set of vectors with $|\delta|$ coordinates being $\epsilon$, and the rest of the coordinates being $0$. 

\begin{prop}{\cite[Ch.~20]{ConwaySloane}}
    With the notations above, the $\hat{\uv}_1, \hat{\uv}_1, \ldots,\hat{\uv}_k$ vectors are 
\[
r(\hat{\up})+\Delta_1, r(\hat{\up})+\Delta_2, r(\hat{\up})+\Delta_3, \ldots.
\]
In particular, the number of top-dimensional components of $\Sigma^0_{\ud}$ is
\begin{equation}\label{eqn:number_top_comp}
    k=\binom{m}{\delta}=\binom{m}{S-m\left\lfloor \frac{S}{m}+\frac{1}{2}\right\rfloor}.
\end{equation}
\end{prop}

\begin{example} \rm \label{ex:closestpoint}
Let $\ud'=(7,7,8,9,12,13)$. Definition~\ref{def:d1} gives $m=3$, and we obtain
\[
\us=(0,-1,-2,-5,-6), \qquad 
\hat{\us}=(0,-1,-2), \qquad
\hat{\up}=\left(\frac{10}{3},\frac{7}{3},\frac{4}{3}\right),
\qquad
r(\hat{\up})=(3,2,1).
\]
We have $\delta=1$, $\epsilon=1$, and the collection of $\Delta$ vectors is
$(1,0,0), (0,1,0), (0,0,1)$.
Therefore $k=3$ and 
$\hat{\uv}_1=(4,2,1)$, $\hat{\uv}_2=(3,3,1)$, $\hat{\uv}_3=(3,2,2)$.
We obtain that \eqref{eq:QIP} takes its minimum at the vectors
\[
\uv_1=(4,2,1,0,0), \qquad
\uv_2=(3,3,1,0,0), \qquad
\uv_3=(3,2,2,0,0).
\]
Further, we have $\hat{D}=\|\hat{\us}-\hat{\uv}_1\|^2=\|\hat{\us}-\hat{\uv}_2\|^2=\|\hat{\us}-\hat{\uv}_3\|^2=34$, and the optimal solution of \eqref{eq:QIP} according to \eqref{eq:gD} is $G_{\ud}^{opt}=39$. We conclude that $\Sigma^0_{\ud'}$ has 3 top-dimensional components, each of codimension $39$.
\end{example}

It remains to compute $\hat{D}$ more explicitly. We write $\left\{x\right\}=x-\lfloor x\rfloor$ for the fractional part of $x\in\R$.

\begin{lemma}\label{lem:delta}
   If $\delta\geq 0$, we have
    \[
    \delta =m\left\{\frac{S}{m}\right\}\quad \text{and}\quad \left\lfloor \frac{S}{m}+\frac{1}{2}\right\rfloor=\frac{S}{m}-\left\{\frac{S}{m}\right\}.
    \]
 If $\delta<0$, we have
\[
\delta = m\left\{\frac{S}{m}\right\}-m\quad \text{and}\quad \left\lfloor \frac{S}{m}+\frac{1}{2}\right\rfloor=\frac{S}{m}-\left\{\frac{S}{m}\right\}+1.
\]
\end{lemma}
\begin{proof}
    Write $S=mT+\bar{S}$ with $T\in \N$ and $0\leq \bar{S}<m$. Then by equation \eqref{eqn:delta}, we have
    \[
    \delta\geq 0 \Leftrightarrow \bar{S}<\frac{m}{2}.
    \]
Moreover,
\[
 m\left\lfloor \frac{S}{m}+\frac{1}{2}\right\rfloor = mT+ m\left\lfloor \frac{2\bar{S}+m}{2m}\right\rfloor 
\]
which together with equation \eqref{eqn:delta} concludes the proof.
\end{proof}

\begin{prop}\label{prop:Dhat_formula}
  We have
  \[
  \hat{D}= m\left\{\frac{S}{m}\right\}\left(1-\left\{\frac{S}{m}\right\}\right)+\frac{S^2}{m}+m d_0^2-2d_0 S
  \]
\end{prop}
\begin{proof}
    We only treat the case $\delta\geq 0$, the other one is analoguous. We have
    \begin{eqnarray*}
        \hat{D} &= &\| \hat{\us}-\hat{\uv}_1 \|^2 \\
        &= & \delta \left(d_0-\left\lfloor \frac{S}{m}+\frac{1}{2}\right\rfloor-1\right)^2+(m-\delta)\left(d_0-\left\lfloor \frac{S}{m}+\frac{1}{2}\right\rfloor\right)^2 \\
        &=& m\left\{\frac{S}{m}\right\}\left(d_0-\frac{S}{m}+\left\{\frac{S}{m}\right\}-1\right)^2+m\left(1-\left\{\frac{S}{m}\right\}\right) \left(d_0-\frac{S}{m}+\left\{\frac{S}{m}\right\}\right)^2 \\
        &=& m\left\{\frac{S}{m}\right\}\left(1-\left\{\frac{S}{m}\right\}\right)+\frac{S^2}{m}+m d_0^2-2d_0 S
    \end{eqnarray*}
where we have used Lemma \ref{lem:delta} in the third equality.
\end{proof}

We have finally arrived at the explicit formulas.

\begin{thm} \label{thm:main-codim}
Let $\ud\in \N^{N+1}$ and $0\leq r\leq \min \ud$. As above, let $\ud'=(d'_0\leq d'_1 \leq \ldots\leq d'_N)\in \N^{N+1}$ be the weakly increasing reordering of $\ud$ and let $m$ be defined as in Definition~\ref{def:d1}. Finally, let $S=\sum_{i=0}^{m}d'_i$ and $\tilde{S}=S-(m+1)r$. We have
\[
    \codim_{\Rep_{\ud}}  \cSigma^r_{\ud} = \frac{m}{2}\left\{\frac{\tilde{S}}{m} \right\}\left(1-\left\{\frac{\tilde{S}}{m} \right\}\right) - \frac{m(m-1)}{2}\left(\frac{\tilde{S}}{m}\right)^2 + \sum_{0\leq i<j\leq m}(d'_i-r)(d'_j-r)
\] 
Let $B\in\Mat_{d_N,d_0}^{\rk=r}$. Then we have
\begin{multline}\label{eqn:codim_fiber}
\codim_{\Rep_{\ud}} \mult^{-1}(B) =  
\frac{m}{2}\left\{\frac{\tilde{S}}{m} \right\}\left(1-\left\{\frac{\tilde{S}}{m} \right\}\right) - \frac{m(m-1)}{2}\left(\frac{\tilde{S}}{m}\right)^2 
\\
+ \sum_{0\leq i<j\leq m}(d'_i-r)(d'_j-r) + r(d_0+d_N-r).
\end{multline}
Both $\cSigma^r_{\ud}$ and $\mult^{-1}(B)$ have the same number $\binom{m}{\tilde{S}-m\left\lfloor \frac{\tilde{S}}{m}+\frac{1}{2}\right\rfloor}$ of top-dimensional irreducible components.
\end{thm}
\begin{proof}
By Lemmas \ref{lem:rank_0} and \ref{lem:rank_vs_fibers}, it suffices to prove the common case $r=0$ of the claims above, namely
\[
    \codim \Sigma^0_{\ud} = \frac{m}{2}\left\{\frac{S}{m} \right\}\left(1-\left\{\frac{S}{m} \right\}\right) - \frac{m(m-1)}{2}\left(\frac{S}{m}\right)^2 + \sum_{0\leq i<j\leq m}d'_i d'_j
\]
which follows from Theorem \ref{thm:QIPreduced2closestpoint}, Proposition \ref{prop:Dhat_formula} and an elementary computation, and
\[
k=\binom{m}{S-m\left\lfloor \frac{S}{m}+\frac{1}{2}\right\rfloor}.
\]
which is equation \eqref{eqn:number_top_comp}.
\end{proof}

\begin{remark} \rm
    It is not immediate from the formulas that those codimensions are integers. However one can easily show that, for any positive integers $\tilde{S},m$, the quantity $\frac{m}{2}\left\{\frac{\tilde{S}}{m} \right\}\left(1-\left\{\frac{\tilde{S}}{m} \right\}\right) - \frac{m(m-1)}{2}\left(\frac{\tilde{S}}{m}\right)^2$ is an integer. It is also possible to rewrite things to make this apparent, but the resulting formula is a bit more complicated and we omit it.
\end{remark}

\begin{example} \rm
    Let $\ud=(d,d,\ldots,d)\in \N^{N+1}$ be a constant dimension vector and $r=0$. We have $m=N$ and $\tilde{S}=S=(N+1)d$. Theorem~\ref{thm:main-codim} then implies
    \[
    \codim \Sigma^0_{\ud} = \frac{N}{2}\left\{\frac{d}{N} \right\}\left(1-\left\{\frac{d}{N} \right\}\right) + \frac{(N+1)}{N}\frac{d^2}{2}.
    \]
   When $N>d$, we have $\left\{\frac{d}{N} \right\}=\frac{d}{N}$ so that
   \[
     \codim \Sigma^0_{\ud}=\frac{d(d+1)}{2}.
   \]
   which is in particular independent of $N$. On the other hand, for fixed $N$ and $d\to\infty$, we see that
   \[
   \codim \Sigma^0_{\ud} \sim_{d\to\infty} \frac{(N+1)}{N}\frac{d^2}{2}.
   \]
    Informally we see in both cases that the magnitude of $\codim \Sigma_{\ud}^0$ is not too sensitive to $N$ (``the depth''), only to $d$ (``the width'') of the quiver (``network''), cf. Section \ref{sec:RLCT}.
\end{example}

\section{Real-log canonical threshold of deep linear networks} 
\label{sec:RLCT}

The initial motivation of this work was to better understand the work of Aoyagi \cite{aoyagi} and to relate it to the theory of quiver representations. Aoyagi's work takes place in the context of \emph{singular learning theory}, the Bayesian statistics of singular statistical models \cite{watanabe2009,watanabe2018,watanabe2024recent}. In singular learning theory, the \emph{real log-canonical threshold (rlct)}, a geometric invariant of singularities of real analytic functions, plays a central role because it controls the asymptotic performance of Bayesian inference.

\begin{defn}\label{defn:rlct}
Let $X$ be a real analytic manifold and $F:X\to \R$ a real analytic function.
\begin{enumerate}[label=(\roman*)]
        \item The \emph{real log canonical threshold} (RLCT) of $F$ is
        \[
        \rlct(F):=\sup \{s\in \R\ |\ |F|^{-s}\text{ is locally integrable }\}\in \R\cup\{\infty\}
        \]
        \item Let $x\in X$. The \emph{local real log canonical threshold} of $F$ at $x$ is
         \[
        \rlct_x(F):=\sup \{s\in \R\ |\ |F|^{-s}\text{ is locally integrable at $x$}\}\in \R\cup\{\infty\}.
        \]
    \end{enumerate}
\end{defn}

The rlct comes with a secondary invariant, the \emph{real log-canonical multiplicity} (rlcm), which we only define in the local case for simplicity. This is an opportunity to introduce archimedean zeta functions, which play an important role in singularity theory and singular learning theory.

\begin{propdefn}
     Let $X$ be a real analytic manifold, $F:X\to \R$ a real analytic function and $x\in X$. Fix a volume form $\mathrm{dvol}$ on $X$ and a relatively compact open neighbourhood $U$ of $x$ The \emph{local archimedean zeta function}
     \[
     \zeta_{F,U}(s):= \int_U |F|^s \mathrm{dvol}
     \]
     which is a priori defined for $\mathrm{Re}(s)\gg 0$ extends to a meromorphic function. The poles of $ \zeta_{F}(s)$ are independent of $U$ small enough and are negative real numbers. The largest pole is $-\rlct_x(F)$, and we define the \emph{real log-canonical multiplicity} $\rlcm_x(F)\in \N$ to be the order of that pole.
\end{propdefn}

This proposition goes back to Atiyah \cite{atiyah} and the proof is based on (real analytic) resolution of singularities.

\begin{remark}\label{rmk:rlct_vs_lct}
As the name suggests, the rlct and the rlcm are the real counterparts of the \emph{log-canonical threshold} (lct) and the associated multiplicity in complex (algebraic and analytic) geometry. The lct plays an important role in complex singularity theory and birational geometry, and has been studied much more intensively than its real counterpart (see \cite{mustata:lct-survey,kollar:lct-survey}.
\end{remark}

Here are some fundamental properties of the rlct which we need to put our results in context. We refer to \cite{lin:phd_thesis} and to the upcoming paper \cite{spl:real-jets} for a comprehensive treatment.

\begin{prop}\label{prop:rlct_elem}
    Let $X$ be a real analytic manifold, $F:X\to \R$ a real analytic function and $x\in X$. 
\begin{enumerate}[label=(\roman*)]
\item $0<\rlct_x(F)<\infty\ \Leftrightarrow\ F(x)=0$ and $\rlct(F)<\infty\ \Leftrightarrow\ F^{-1}(0)\neq\emptyset$. 

Moreover, when $F(x)=0$, we have
\begin{equation}\label{eqn:rlct_upper_bound}
    \rlct_{x}(F)\in \left(0, \frac{\codim_{X,x} F^{-1}(0)}{2}\right]\cap \Q.
\end{equation}
\item The function $x\mapsto\rlct_{X,x}(F)$ is lower semi-continuous.
\item \[
\rlct(F)=\inf_{x\in X}\rlct_{x}(F)
\]
This inf is not always attained in the general real analytic case because of potential issues ``at infinity", but it is in many cases f.(e.g., when $X$ is compact or when $X,F$ are algebraic). When it is attained, we have 
\begin{equation}\label{eqn:rlct_upper_bound_glob}
    \rlct(F)\in \left(0, \frac{\codim_X F^{-1}(0)}{2}\right]\cap \Q.
\end{equation}
\item\label{thom-seb} Let $G:X\to \R$ be another real analytic function. Then we can consider $F+G$ and $F\cdot G$ as functions on $X\times Y$. Then
\begin{eqnarray*}
\rlct(FG)  =& \min(\rlct(F),\rlct(G))\quad \text{ and }
\\ \rlcm(FG) =&\left\{\begin{array}{cl}
  \rlcm(F)   & \text{if}\ \rlct(F)<\rlct(G) \\
   \rlcm(G)  & \text{if}\ \rlct(F)>\rlct(G) \\
   \rlcm(F)+\rlcm(G)  & \text{if}\ \rlct(F)=\rlct(G).
\end{array}\right.
\end{eqnarray*}
Assume $F,G\geq 0$. Then
\[
\rlct(F+G)=\rlct(F)+\rlct(G)\quad \text{ and }\quad \rlcm(F+G)=\rlcm(F)+\rlcm(G)-1
\]
\end{enumerate}
\end{prop}

\begin{example}\label{ex:rclts} \rm
The definition immediately implies, for any $n\in \N$
\[
\rlct(x^n)=\frac{1}{n}.
\]
Using Statement \ref{thom-seb} above, we deduce that, for $X=\R^d$ and $e\leq d$
\[
\rlct(x_1^{n_1}\ldots x_d^{n_e})= \min \frac{1}{n_i}
\]
and
\[
\rlct(x_1^{2n_1}+\ldots+x_e^{2n_e})=\frac{1}{2n_1}+\ldots+\frac{1}{2n_e}
\]
In particular, if $F=x_1^{2}+\ldots+x_e^{2}$ then
\[
\rlct(F)=\frac{e}{2}=\frac{\codim_X F^{-1}(0)}{2}
\]
saturates the inequality \eqref{eqn:rlct_upper_bound_glob}.
\end{example}

In singular learning theory, the rlct and the rlcm of a certain function $K$ control the asymptotic performance of Bayesian inference in a large class of statistical models as the size of the dataset increases. More precisely, $K$ arises as the \emph{relative entropy (or Kullback-Leibler divergence) between the true distribution and the model}, considered as a function of model parameters. 

In \cite{aoyagi}, Aoyagi computes the real log-canonical threshold of the relative entropy function of \emph{deep linear neural networks}. Deep linear networks are obtained from standard (feedforward, fully-connected) deep linear networks by replacing their non-linear activation functions by linear maps. Despite their simplicity, they are a useful ``toy model" in modern deep learning theory, as we discussed in Section \ref{sec:related}. By definition, the weights of a deep linear networks are a tuple of composable matrices, and the function computed by the network is their product, so the parameter space of the model is $\Rep_{\ud}$ where $\ud\in \N^{N+1}$ records the widths of the layers and $N$ is the depth of the network. We also need a ``true distribution" generating the data, or rather in this context a ``true function". Aoyagi makes the assumption that this true function is linear and given by some $B\in \Mat_{d_N,d_0}^{\rk=r}$ with $0\leq r\leq \min{\ud}$. This condition on the rank emsures that $B\in \mathrm{Im}(\mult)$; in statistical terminology, we are in the \emph{realisable} or \emph{well-specified} case.

We now define directly the function $K=K^{\DLN}_B$ for our deep linear model, referring to \cite{aoyagi} for a derivation of how it arises as a relative entropy. As it turns out, in this case, $K^{\DLN}_B$ is algebraic:
\begin{equation}
    K^{\DLN}_B(A_*):= \| \mult(A)-B \|^2_2 =\Tr((\mult(A)-B)^t (\mult(A)-B))) 
\end{equation}
We have $K^{\DLN}_B(A_*)\geq 0$ and 
\[(K^{\DLN}_B)^{-1}(0)=\mult^{-1}(B)\]
so that by Equation \ref{eqn:rlct_upper_bound} we find
\[
\rlct(K^{\DLN}_B)\leq \frac{\codim \mult^{-1}(B)}{2}.
\]

The main result of \cite{aoyagi} is then the computation of $\rlct(K^{\DLN}_B)$, which thanks to our results we can reformulate cleanly as

\begin{thm}\label{thm:aoyagi-rlct}
\[
\rlct(K^{\DLN}_B)= \frac{\codim \mult^{-1}(B)}{2}.
\]
and, with the notations of Section~\ref{sec:SolvingQIP},
\[
\rlcm(K^{\DLN}_B)=m^2\left\{\frac{\tilde{S}}{m} \right\}\left(1-\left\{\frac{\tilde{S}}{m} \right\}\right)
\]
\end{thm}
\begin{proof}
    This follows from a comparison between Formula \eqref{eqn:codim_fiber} in Theorem \ref{thm:main-codim} and \cite[Theorem 1]{aoyagi}. To see this requires some translations between our notations and the ones in loc.cit (Note the shift by one and reverse ordering of the dimensions):
\[
\begin{array}{c|c|c}
    & \text{\cite{aoyagi}} &\text{This paper} \\
 \hline  \text{Dimension vector}  & (H^{(L+1)},\ldots,H^{(1)})  & (d_0,\ldots,d_N)  \\
   \text{Reduced dimension vector}  &  (M^{L+1},\ldots,M^{(1)}) & (d_0-r,\ldots,d_N-r) \\
   \text{Number of ``relevant" dimensions} & l & m \\
   \text{Set of ``relevant" reduced dimensions} & (M^{(S_1)},\ldots,M^{(S_{l+1})}) & (d'_0-r,\ldots,d'_m-r) \\
\end{array}
\]

In particular, the integers $M$ and $a$ defined in \cite[Theorem 1]{aoyagi} translate in our notation to
\[
M=\left\lceil \frac{\tilde{S}}{m}\right\rceil
\]
and
\[
a=\tilde{S}-m\left\lceil\frac{\tilde{S}}{m}\right\rceil+m
\]
so that
\[
\frac{a(l-a)}{l}=m\left (1+\frac{\tilde{S}}{m}-\left\lceil\frac{\tilde{S}}{m}\right\rceil\right)\left (-\frac{\tilde{S}}{m}+\left\lceil\frac{\tilde{S}}{m}\right\rceil\right)
\]
Now we make the following elementary observation about ceiling and floor functions: for all $x\in \R$, 
    \[
    (\left\lceil x\right\rceil-x)(1+x-\left\lceil x\right\rceil)=\left\{x\right\}(1-\left\{x\right\}).
    \]
and conclude that
\[
\frac{a(l-a)}{l}=m\left\{\frac{\tilde{S}}{m} \right\}\left(1-\left\{\frac{\tilde{S}}{m} \right\}\right).
\]
Using this, we now recognize that Formula \eqref{eqn:codim_fiber} in Theorem \ref{thm:main-codim} is exactly (twice) the first formula for $\lambda=\rlct(K^{\DLN}_B)$ in \cite[Theorem 1]{aoyagi}. There is also a formula for the multiplicity in \cite[Theorem 1]{aoyagi}, which we simply translate into our notations. 
\end{proof}

\begin{remark} \rm
As far as we now, there is no simple relationship between $\rlcm(K^{\DLN}_B)=m^2\left\{\frac{\tilde{S}}{m} \right\}\left(1-\left\{\frac{\tilde{S}}{m} \right\}\right)$ and the number $k=\binom{m}{\tilde{S}-m\left\lfloor \frac{\tilde{S}}{m}+\frac{1}{2}\right\rfloor}$ of irreducible components of $\mult^{-1}(B)$.
\end{remark}

We expect that the formula of Theorem~\label{thm:aoyagi-rlct} for global rlcts also extends to \emph{local} rlcts and codimensions, i.e., that for every $A_*\in \mult^{-1}(B)$ we have $\rlct_{A_*}(K^{\DLN}_B)=\frac{\codim_{A_*} \mult^{-1}(B)}{2}.$ We plan to come back to this question in future work.

\appendix

\section{Another algebraic characterization of $C$ and $\theta$.} \label{sec:app_theta}

We showed three efficient ways to calculate the dimension of top-dimensional components of $\Sigma^r_{\ud}$: Theorems \ref{thm:Pseries}, \ref{thm:QIP}, and \ref{thm:main-codim}. In this section---as an added bonus---we give another characterization (of both $C$ and $\theta$) that is less calculationally efficient than the other three. However, it suggests some relations with the so-called interpolation method for characteristic classes of singularities.

For a dimension vector $\ud$ and an integer $r$, consider the degree 1 variables $y_1,\ldots,y_{r+1}$ and $x_{ij}$ for $0\leq i \leq N, 1\leq j \leq d_i$. 
Define 
\[\Rr_{\ud}=\Z[x_{ij}: 0\leq i \leq N, 1\leq j \leq d_i],
\qquad\text{ and }\qquad 
\Hh_{\ud}= \Rr_{\ud}^{S_{\ud}}
\]
where $S_{\ud}=S_{d_0}\times S_{d_1}\times \ldots \times S_{d_N}$, and the symmetric group $S_{d_i}$ permutes the variables $x_{i1},\ldots,x_{id_i}$.  
The map $\psi_r$ in the diagram
\[
\begin{tikzcd}
\Rr_{\ud} \arrow[r, "\psi_r"] &
\Z[y_1,\ldots,y_{r+1}, x_{ij}: 0\leq i \leq N, r+1< j \leq d_i] \\
\Hh_{\ud} \arrow[u, hook] \arrow[ru, "\phi_r"] & 
\end{tikzcd}
\]
is defined by 
\[
x_{ij}\mapsto 
\begin{cases}
    y_j & \text{ if } j\leq r+1\\
    x_{ij} & \text{ if } j>r+1.
\end{cases}
\]
The map $\phi_r$ is the restriction of $\psi_r$ to $\Hh_{\ud}$. Let $\CA^r_{\ud}$ be the kernel of $\phi_r$---a homogeneous ideal.

As before, the number and codimension of the top-dimensional components of $\Sigma_{\ud}^r$ are denoted by $\theta$ and $C$.

\begin{thm}\label{thm:UsingAvoidingIdeal}
    The degree and rank of the lowest degree part of $\CA^r_{\ud}$ are $C$ and $\theta$.  
\end{thm}

\begin{proof}
Here we just sketch the proof. The elements of $H_G^*(\Rep_{\ud})(\cong H_{\ud})$ that are supported on $\Sigma_{\ud}^r$ form an ideal, the so-called {\em avoiding ideal}. 
On the one hand, this ideal can be calculated as a kernel of the restriction map to $\Rep_{\ud}-\Sigma_{\ud}^r$ (namely, the map $\phi_r$). On the other hand, the fundamental classes of the largest dimensional components of $\Sigma_{\ud}^r$ form a basis for the lowest-degree part of $\CA_{\ud}^r$, and the theorem follows. Details on avoiding ideals in general are in \cite{FRavoiding}, and details on restriction maps for $\Rep_{\ud}$ are in \cite{FRduke}.
\end{proof}

\begin{example}\rm (Cf. Examples \ref{ex:222324}, \ref{ex:q-series}, \ref{ex:223withQIP}.)
For $\ud=(2,2,3)$ let us use the following notation: let $a_1,a_2$; $b_1,b_2$; $c_1,c_2,c_3$ be the elementary symmetric polynomials of $x_{11},x_{12}$; $x_{21},x_{22}$; $x_{31},x_{32},x_{33}$, respectively. Hence $\deg(a_i)=\deg(b_i)=deg(c_i)=i$. By definition, the ideal $\CA^0_{\ud}$ is the kernel of the ring homomorphism
\[
\begin{array}{rcl}
    \phi_0:\Z[a_1,a_2,b_1,b_2,c_1,c_2,c_3] & \to & \Z[y,x_{12},x_{22},x_{32}, x_{33}] \\
    x_{11}, x_{21}, x_{31} & \mapsto & y.
\end{array}
\] 
Calculation shows that the lowest degree part of this ideal is of degree 4, and of rank~2. In fact, it is spanned by the {\em two} ($\Rightarrow \theta=2$) degree {\em four} ($\Rightarrow C=4$) polynomials
\[
    a_1b_1b_2 - a_1b_2c_1 - a_2b_1^2 + a_2b_1c_1 + a_1c_3 + a_2b_2 - a_2c_2 - b_1c_3 - b_2^2 + b_2c_2,\text{   and}
\]
\begin{multline*}
    \hskip .8 true cm a_1^2b_2 - a_1a_2b_1 - a_1b_1b_2 + a_2b_1^2 + a_2^2 - 2a_2b_2 + b_2^2 \\
    =(x_{21}-x_{11})(x_{22}-x_{11})(x_{21}-x_{12})(x_{22}-x_{12}),
\end{multline*}
that are the fundamental classes of the orbit closures corresponding to the lace diagrams
\begin{equation*} 
    \begin{tikzpicture}[baseline=-5]
 \node at (0,0)  {$\bullet$};
 \node at (0,-.2)  {$\bullet$};
 \node at (1,0)  {$\bullet$};
 \node at (1,-.2)  {$\bullet$};
 \node at (2,0)  {$\bullet$};
 \node at (2,-.2)  {$\bullet$};
 \node at (2,-.4)  {$\bullet$};
  \draw (1,0) -- (2,0);
  \draw (0,-.2)--(1,-.2);
\end{tikzpicture}
\qquad \text{ and }\qquad
\begin{tikzpicture}[baseline=-5]
 \node at (0,0)  {$\bullet$};
 \node at (0,-.2)  {$\bullet$};
 \node at (1,0)  {$\bullet$};
 \node at (1,-.2)  {$\bullet$};
 \node at (2,0)  {$\bullet$};
 \node at (2,-.2)  {$\bullet$};
 \node at (2,-.4)  {$\bullet$};
  \draw (1,0) -- (2,0);
  \draw (1,-.2)--(2,-.2);
\end{tikzpicture}.
\end{equation*}
\end{example}

Clearly, Theorem \ref{thm:UsingAvoidingIdeal} is less computationally efficient than our earlier theorems calculating $C$ and $\theta$. Nevertheless, it (just like Theorem \ref{thm:Pseries}) also displays the a priori non-obvious fact that $C$ and $\theta$ are invariant under the permutation of the components of $\ud$ (cf.~Corollary~\ref{cor:PermutationInvariance}).

\bibliographystyle{alpha}
\bibliography{references}

\end{document}